\documentclass[twoside]{article}

\usepackage{xcolor,booktabs}
\usepackage{float}
\usepackage{comment}
\usepackage{amsmath, amsfonts, amsthm, amssymb, cancel}
\usepackage{dsfont}
\usepackage{enumitem}
\usepackage{graphicx} 
\usepackage[pagebackref=true]{hyperref}
\hypersetup{
    colorlinks=true,
    linkcolor=blue,
    filecolor=magenta,     
    citecolor = teal, 
    urlcolor=cyan,
    pdfpagemode=FullScreen,
    }
\usepackage{geometry}
 \usepackage{orcidlink}
 \usepackage{authblk}
 \usepackage{fancyhdr}
 \usepackage{titling}
\usepackage{subcaption}
\usepackage{multirow}
\usepackage{adjustbox}
\usepackage{array}

\usepackage{caption}

\newcommand{\E}{\mathbb{E}}
\newcommand{\R}{\mathbb{R}}

\newcommand{\Eq}{\uu_{\text{eq}}}

\newcommand{\uu}{{u}}
\newcommand{\dt}{\Delta t}
\newcommand{\TT}{\mathbf{T}}

\renewcommand{\P}{\mathbb{P}}

\newcommand{\ind}[1]{\mathds{1}_{#1}}

\DeclareMathOperator{\Law}{Law}

\newcommand{\cotaMin}{\vartheta_0}
\newcommand{\quo}{\mathcal{Q}}
\newcommand{\varInfty}{\mathbb{V}_\infty}
\newcommand{\Tol}{\text{Tol}}
\renewcommand{\P}{\mathbb{P}}

\newtheorem{theorem}{Theorem}[section]
\newtheorem*{theorem*}{Theorem}
\newtheorem{definition}[theorem]{Definition}
\newtheorem{proposition}[theorem]{Proposition}

\newtheorem{lemma}[theorem]{Lemma}
\newtheorem{assumption}[theorem]{Assumption}
\theoremstyle{remark}
\newtheorem{remark}[theorem]{Remark}

\numberwithin{equation}{section}

\title{Expectation–Maximization algorithm to estimate the forcing parameter of a nonlinear McKean--Vlasov diffusion}

\author{Eduardo Gutiérrez-Turner\orcidlink{0000-0002-2983-0513}}
\affil{Instituto de Estadística, Universidad de Valparaíso. eduardo.gutierrezt@postgrado.uv.cl}
\affil{Universidad Adventista de Chile. eduardogutierrez@unach.cl}

\author{Kerlyns Mart\'inez\orcidlink{0000-0001-5277-6618}}
\affil{Departamento de Ingeniería Matemática, Universidad de Concepción. kermartinez@udec.cl}

\author{Héctor Olivero\orcidlink{0000-0003-3589-0130}}
\affil{Cimfav - Ingemat, Universidad de Valparaíso. hector.olivero@uv.cl}
\date{\today}

\begin{document}
\pagestyle{fancy}
\fancyhead{}
\fancyhead[OC]{Gutiérrez-Turner et al.}
\fancyhead[EC]{ EM algorithm to estimate the forcing parameter of a nonlinear McKean--Vlasov diffusion}
\maketitle
\abstract{
 In this article, we address the problem of estimating a forcing parameter in a stochastic differential equation inspired by a model that describes instantaneous turbulent kinetic energy. The stochastic differential equation we analyze is of the nonlinear McKean–Vlasov type, where the drift term depends on a power of the expected value of the solution, which also introduces nonlinearity in an algebraic sense. We propose an estimation algorithm based on the Expectation–Maximization framework and show the consistency of our method. %Additionally, 
 We illustrate our findings through numerical experiments.
}

\medskip
{\small \noindent \textbf{MSC 2020:} Primary 62M99; secondary 62F12;  }
%%%%%%%%%%%%%%%%%%%%%%%%%%%%%%%%%%%%%
%%%%%%%%%%% BODY %%%%%%%%%%%%%%%%%%%

% !TEX root = main.tex

\section{Introduction}

Nonlinear McKean--Vlasov (MV) diffusions are stochastic differential equations where the coefficients depend not only on time and the state of the system but also on the law of the process itself. Originally introduced by McKean \cite{McKean1966} in the context of gas kinetic theory, McKean--Vlasov equations emerge as the limiting equations for systems of interacting particles/individuals when their number goes to infinity (see \cite{Sznitman1991,Meleard:1996aa,ChaintronDiez2022a,ChaintronDiez2022b}). This formalism, which is known as propagation of chaos, allows one to derive global models starting from local interactions, and it has been applied in very different contexts such as ecology \cite{videla2025persistence}, stochastic modelling in fluid mechanics \cite{pope1994lagrangi}, game theory \cite{CarmonaDelarue2018,CarmonaDelarue2018b}, neuroscience \cite{Bossy:2019aa, quininao2020clamping, crevat2019mean,colombani2023propagation} and population genetics \cite{Cordero:2026aa,lambert2026evolution}, to name a few. We are interested in a parametrized MV SDEs with the structure
\begin{equation}\label{eq:generic-MVSDE}
    dx_t = \Phi(\theta,t,x_t,\mu_t)dt + \Sigma(t,x_t,\mu_t)dB_t,
\end{equation}
where $\mu_t = \Law(x_t)$, $B$ is a Brownian motion and $\theta$ is a parameter, possibly multidimensional. The mathematical analysis when $\Phi$ and $\Sigma$ are Lipschitz with linear growth is now classical, and in recent years, well-posedness and numerical approximation have been addressed for coefficients with less regularity and super-linear growth. See for example \cite{chen2025,kumar2021explicit,kumar2022well}.

%In terms of statistical inference, although for SDEs in the classical It\^o setting is now well established (see Kessler \cite{Kessler1997}, Kutoyants \cite{Kutoyants2004}, S\o rensen \cite{Sorensen2004} and Fuchs \cite{Fuchs2013}), inference for $\theta$ in \eqref{eq:generic-MVSDE} have been addressed by several authors in recent years. 
Statistical inference for $\theta$ in \eqref{eq:generic-MVSDE} have been addressed by several authors in recent years. Under continuous-time observation, and for linear dependency of $(\Phi,\Sigma)$ in $\mu_t$, Wen et al. \cite{Wen2016} propose a maximum likelihood estimator (MLE) under Lipschitz assumptions on $\Phi$, with $\Sigma \equiv 1$; Genon-Catalot and Laredo \cite{GenonCatalotLaredo2021} study the case when \eqref{eq:generic-MVSDE} is one-dimensional with polynomial growth in the state variable; and Liu and Qiao \cite{LiuQiao2022} study an MLE when $\Phi$ and $\Sigma$ are path-dependent, with linear growth, but not Lipschitz. Sharrock et al. \cite{Sharrock2023} consider the case where $\Sigma$ depends only on the state variable $x_t$, and under suitable hypotheses on $\Phi$ the authors propose online and offline estimators built from samples of several realizations of either \eqref{eq:generic-MVSDE} or its associated particle system. 
Under discrete-time observation, and Lipschitz dependency  of $(\Phi,\Sigma)$ in $\mu_t$, Ren and Wu \cite{RenWu2021} address the path-dependent case with Lipschitz coefficients, and propose a least-squares-type estimator; meanwhile, Amorino et al. \cite{Amorino2023} address joint estimation of drift and diffusion coefficients from discrete observations of the particle system associated with \eqref{eq:generic-MVSDE}. As we will see below, in our case, $\Sigma$ is linear in the state variable, but $\Phi$ has a nonlinear dependency in the law of the process. In addition, we will assume that we are only observing one trajectory of \eqref{eq:generic-MVSDE}, and that we do not have access to its associated particle system. Therefore, we are outside the framework of all these works mentioned above.

In this article, we introduce a type of expectation-maximization algorithm for estimating a forcing parameter in a McKean--Vlasov diffusion, in which the drift depends on the law of the process through a power of its first moment. Our method is developed to work with time-discrete observations of a single trajectory of the model over a finite time horizon $[0,T]$, making it suitable for experimental data. At each E-step, the nonlinear mean-field term is approximated by its current estimate, yielding a linear SDE that is then used in the M-step to update the parameter estimate. Moreover, its computational cost per iteration is linear in the size of the data. In addition, it converges in a moderate number of iterations on synthetic data. In contrast to previous approaches, we do not make use of likelihood representations through Girsanov-type transformations, which do not seem easy to generalize to our context. 

Closer to our approach, Pavliotis and Zanoni proposed to linearize the nonlinear McKean SDE by replacing the law of the process with its (unique) invariant measure, reducing the inference problem to a more tractable setting in which the parameter can be estimated from a single observed particle as the number of particles and the time horizon tend to infinity. Building on this idea, they have developed several estimators along these lines, including a martingale estimator based on the eigenvalues and eigenfunctions of the generator of the linearized mean field limit \cite{Pavliotis2022}; a method-of-moments estimator for systems with polynomial drift and interaction functions, in which the moments of the invariant distribution are approximated by time-averages of the observed trajectory \cite{Pavliotis2024}; and, more recently, a linearized maximum likelihood estimator for the resulting stationary process \cite{Pavliotis2025}. Despite this proximity in approach, our setting differs from theirs in several respects. First, in Pavliotis and Zanoni's work, stationarity is a structural requirement, while in our setting the algorithm operates in the transient regime, treating the mean path as a time-dependent unobserved quantity to be recovered over a finite horizon $[0,T]$; for us, stationarity enters only as a technical device within the proof of consistency as $T \to \infty$. Second, while the McKean SDE in Pavliotis and Zanoni's work arises only as the $N \to \infty$ mean field limit of an underlying particle system, in our setting the McKean-type SDE \eqref{eq:SM1} is itself the observation model. Third, although our model \eqref{eq:SM1} is considerably simpler in other respects, its mean-field coupling lies outside the class of interactions considered in \cite{Pavliotis2022, Pavliotis2024, Pavliotis2025}: the drift in \eqref{eq:SM1} includes a power of the first moment, while the methods in \cite{Pavliotis2022, Pavliotis2024, Pavliotis2025} are built on pairwise interaction kernels acting through convolution with the density of the solution, which can only be expanded into linear combinations of moments rather than powers of them.

\paragraph{Model}
We consider a fixed filtered probability space $(\Omega, \mathcal{F}, (\mathcal{F}_t)_{t \geq 0}, P)$, which is assumed to satisfy the usual conditions. Let $W = (W_t)_{t \geq 0}$ be a standard Brownian motion defined on this space. We focus on the following stochastic differential
equation, parametrized by $\theta \in \mathbb{R}_+$:

\begin{equation}\label{eq:SM1}
    X_t = X_0 + \int_0^t \big(\theta-c(\E[X_s])^{p-1}X_s\big)ds + \int_0^t\sigma X_sdW_s,
\end{equation}
where $c,\sigma \in \mathbb{R_+}$, and $p>1$ are known. In the rest of the paper we assume that $\E[X_0]<\infty$.

This equation is inspired by the reduced model for the instantaneous turbulent kinetic energy proposed by the authors in \cite{Bossy2022}, where the dependence on the distribution is encoded via powers of the first marginal moment. Notice, as in the Ornstein-Uhlenbeck process, if the forcing parameter $\theta=0$, the mean of the process goes to $0$ as $t$ goes to infinity. 

In the following proposition we establish the basic properties of the model. We postpone its proof to Appendix \ref{app:preliminaries_proofs}. 

\begin{proposition}\label{prop:well-posedness-and-properties-NLSDE}
Let $X_0\in\R_+$ be a positive initial condition, independent of $W$, with finite expectation and let $T>0$ be an arbitrary time horizon. Then: 
\begin{enumerate}[label=(\roman*)]
\item (Well-posedness and possitivity) The McKean-Vlasov SDE \eqref{eq:SM1} admits a unique positive strong solution $(X_t)_{0\le t \le T}$.  \label{Pro:WellPosedness}
\item\label{pro:Finite Moments} 
    The solution to the MVSDE \eqref{eq:SM1} satisfies:
    \begin{enumerate}
        \item {\bf Moment control in finite time intervals}: For all $\gamma \ge0$, and finite time horizon $T>0$, if $X_0\in L^\gamma(\Omega)$, then
        \begin{equation}\label{eq:uniform-moments-in-finite-intervals}
            \E\left[\sup_{0\le s\le T} X_s^{-\gamma}\right]+\E\left[\sup_{0\le s\le T} X_s^{\gamma}\right] < \infty. 
        \end{equation}
        \item {\bf Uniform control of the first moment:}
        \begin{equation}\label{eq:uniform-first-moment}
            \sup_{ t\geq 0}\E\left[ X_t\right] < \infty.
        \end{equation}

    \end{enumerate}
\end{enumerate}

\end{proposition}

\paragraph{Main result}

The proposed method can be interpreted as an Expectation--Maximization (EM) algorithm with two stages: first, the estimation of the first marginal moment, and second, a maximum likelihood update of the parameter. This procedure defines an estimator $\hat{\theta}^{(k,T,\Delta t)}$ for the forcing parameter $\theta$, whose consistency is established in the following result as the sample size, time horizon, and number of EM iterations tend to infinity.

\begin{theorem*}[Consistency of the EM algorithm estimator]
Under Assumption \ref{hip:sign_theta} bellow, for every $\epsilon>0$,
\begin{equation*}
\lim_{k\to\infty}\lim_{T\to\infty}\lim_{\dt\to0} \P\left(\left|\hat{\theta}^{(k,T,\dt)}-\theta\right| > \epsilon \right) =0.
\end{equation*}

\end{theorem*}

The remainder of the paper is organized as follows: In Section \ref{sec:algorithm_convergence} we outline the estimation procedure and establishes the consistency of the proposed statistic. In Section \ref{sec:numerical_experiments} we illustrate our findings through numerical experiments. In Section \ref{sec:closing} we summarize our results, and propose some posible extensions. Finally in the Appendix \ref{sec:appendix} we include the most technical proofs.

\section{Estimation Method and Convergence Analysis}\label{sec:algorithm_convergence}

In this work, we propose a methodology for estimating the parameter $\theta$. The parameter $c>0$ is assumed to be known due to identifiability issues in the drift structure, while the estimation of $\sigma$ can be carried out independently via a standard quadratic variation approach. Indeed, the solution to \eqref{eq:SM1} satisfies
\[
\langle X \rangle_T = \sigma^2 \int_0^T X_t^2 \, dt.
\]
Under high-frequency observations of $X$, this leads to the estimator
\[
\hat{\sigma}_n^2
=
\frac{\sum_{i=0}^{n-1}(X_{t_{i+1}}-X_{t_i})^2}
{\sum_{i=0}^{n-1} X_{t_i}^2 (t_{i+1}-t_i)},
\]
whose consistency and asymptotic normality are classical (see, e.g., \cite{dohnal1987estimating, Miao2004}).

%In the proofs throughout the paper, the parameter $\sigma$ is assumed to be known, while the estimation proposed for $\theta$ is carried out independently of this parameter.

In the proofs throughout the paper, since the estimation proposed for $\theta$ is carried out independently of the parameter $\sigma$, we will  assume it to be known.

\subsection{Moment-EM Algorithm} 

We assume that the process $X$ is observed on a regular time grid over a finite time horizon $[0,T]$. More precisely, we observe the process at discrete times over $[0,T]$,
\[
\mathcal{D}=\{X_{t_i}\}_{i=0}^n, \qquad 0=t_0<t_1<\cdots<t_n=T,
\]
where the observation grid has size $\Delta t = \max_{1\le i\le n}(t_{i} - t_{i-1})$. We assume we also observe the initial mean $u_0 := \E[X_0]$.

An Expectation--Maximization type algorithm is proposed to estimate the forcing parameter $\theta$. The procedure iterates between two steps: an E-step, which approximates the unobserved mean path, and an M-step, which updates the estimate of $\theta$ via maximum likelihood based on this first marginal moment. EM-type algorithms have also been applied to other diffusion estimation problems with a different latent-variable structure, such as the integrated and noise-contaminated observations considered in \cite{Baltazar2010}, where the unobserved quantity is the continuous-time path of the diffusion rather than its mean path.

We initialize the algorithm by considering an initial value $\hat{\theta}^{(0,T, \dt)} =0$. Then, at each iteration $k$ we follow:

\subsubsection{E-Step: Solving the Mean ODE}
We compute an approximation of the mean function, $(\uu_t^{(k, T, \Delta t)})_{0\le t\le T}$, using the most recent parameter estimate, $\hat{\theta}^{(k-1,T,\dt)}$. This is achieved by numerically solving the non-linear ODE:
$$ \dfrac{d}{dt}u^{(k,T,\dt)}_t = \hat{\theta}^{(k-1,T,\dt)} - c (u^{(k,T,\dt)}_t)^p, \quad u_0=\E[X_0],$$
over the time interval $[0, T]$. 

\begin{remark}
We assume that the numerical error associated with the approximation of the ODE solution in this step is negligible, as it can be computed with sufficiently high accuracy using standard numerical methods.
\end{remark}

\subsubsection{M-Step: Pseudo-Maximum Likelihood Estimator}
The M-step consists of updating the parameter $\theta^{(k,T,\Delta t)}$ by maximizing the likelihood of the observed data under model \eqref{eq:SM1}, given the current approximation of the mean path. 

By fixing the current mean path, the model reduces to a linear SDE, for which the transition density can in principle be identified explicitly. However, for simplicity and computational efficiency, we instead maximize a pseudo-likelihood constructed from the Euler–Maruyama discretization. This approximation is known to provide reliable estimators under high-frequency observations and is widely used in the statistical inference of diffusion processes.

The discretization implies that the increment $\Delta X_i:=X_{t_{i+1}} - X_{t_i}$ is approximately normally distributed\[\Delta X_i \,|\, \mathcal{F}_{t_i} \sim \mathcal{N}\Big((\theta - c(u_{i}^{(k)})^{p-1}X_i)\Delta t_i,\quad  \sigma^2 X_i^2 \Delta t_i\Big).\]

Maximizing the corresponding log-pseudo-likelihood function with respect to $\theta$ in the parameter space $\Theta = \R_+$, yields the updated estimator for the $(k+1)$-th iteration: 
\begin{equation}
    \label{eq:SM1 Estimador MV}
    \hat{\theta}^{(k,T,\dt)} =  \max\left\{\quo(T,\dt)+ c\frac{\sum_{i=0}^{n-1} \frac{({u}_{t_i}^{(k,T,\dt)})^{p-1}}{X_{t_i}}\Delta t_i}{  \sum_{i=0}^{n-1} X_{t_i}^{-2} \Delta t_i},\quad 0\right\},
\end{equation}
where $\quo(T,\dt):= \frac{\sum_{i=0}^{n-1}X_{t_i}^{-2}\Delta X_{t_i} }{  \sum_{i=0}^{n-1} X_{t_i}^{-2} \Delta t_i}$, depends only on the data.

A key characteristic of the moment-EM algorithm is that it preserves positivity.

\begin{lemma}\label{prop:basics-of-true-estimator}
Assume there exist $T_0>0$ and $\dt_0>0$ such that $\hat{\theta}^{(1,T_0,\dt_0)}>0$. Then, the sequence $\{\hat{\theta}^{(k,T_0,\dt_0)}\}_{k\ge1}$ is increasing and strictly positive. 
\end{lemma}

\begin{proof} 
Define
\[
F_k := \quo(T_0,\dt_0)+ c\frac{\sum_{i=0}^{n-1} \frac{({u}_{t_i}^{(k,T_0,\dt_0)})^{p-1}}{X_{t_i}}\Delta t_i}{\sum_{i=0}^{n-1} X_{t_i}^{-2} \Delta t_i},
\]
so that $\hat{\theta}^{(k,T_0,\dt_0)} = \max\{F_k,\ 0\}.$

From Lemma \ref{prop:ode_properties} (iii), if
\[
\hat{\theta}^{(k-1,T_0,\dt_0)}<\hat{\theta}^{(k,T_0,\dt_0)},
\]
then for all $t_i>0$,
\[
{u}_{t_i}^{(k,T_0,\dt_0)}<{u}_{t_i}^{(k+1,T_0,\dt_0)}.
\]

Since $p\ge1$, $c>0$ and ${u}_{\cdot}^{(k,T_0,\dt_0)}$ and $X$ are strictly positive:
\[
F_k \le F_{k+1}.
\]
Thus, for any $k\geq1$
\[
\hat{\theta}^{(k,T_0,\dt_0)} = \max\{F_k,\ 0\} \le \max\{F_{k+1},\ 0\} = \hat{\theta}^{(k+1,T_0,\dt_0)},
\]
i.e. $\{\hat{\theta}^{(k,T_0,\dt_0)}\}_{k\ge0}$ is non-decreasing.

Moreover, if $\hat{\theta}^{(1,T_0,\dt_0)}>0$, then $F_1>0$. Hence, for all $k\ge1$,
\[
F_k \ge F_1 > 0 \;\Rightarrow\; \hat{\theta}^{(k,T_0,\dt_0)} = F_k > 0.
\]
\end{proof}

In practice, the observed trajectory is fixed, which also fixes the frequency  $\dt$ and terminal time $T$. This means that if, for some iteration $k\ge1$, we find $ \hat{\theta}^{(k,T,\dt)} = 0 $, then it follows that $ \hat{\theta}^{(m,T,\dt)} = 0$ for all $ m \geq k $. This outcome is, of course, undesirable so we consider the following assumption:

\begin{assumption}\label{hip:sign_theta}
Assume that, for the observed trajectory $X(\omega)$, with observation frequency $\dt$ and horizon time $T$, we have \[\hat{\theta}^{(1,T,\dt)}(\omega)>0.\]
\end{assumption}

Under Assumption \ref{hip:sign_theta}, we ensure from Lemma \ref{prop:basics-of-true-estimator} the estimator remains positive throughout all iterations of the algorithm. It can be shown however that this assumption holds in probability for sufficiently small frequency and sufficiently large time horizon.

\begin{lemma}\label{cor:the-estimator-is-eventually-positive}
For any $k\ge1$, we have
$$
\lim_{T\to\infty}\lim_{\dt\to0}\P\left(\hat{\theta}^{(k,T,\dt)}   >  0 \right)  = 1.
$$
\end{lemma}
\begin{proof}
Since $\hat{\theta}^{(k-1,T,\dt)}\geq0$, we notice that
\[\hat\theta^{(k,T,\dt)} \ge \quo(T,\dt).\]

Hence, $\P\left(\hat{\theta}^{(k,T,\dt)} \le  0 \right) \le \P\left(\quo(T,\dt) \leq  0 \right)$. So we focus on the estimation of the later probability. 

First, from the properties of stochastic integrals (see, e.g., \cite[Sec. I.7, II.4]{Protter1992}), we have
\[\quo(T,\dt)\rightarrow \quo(T):= \frac{\int_0^T\frac{dX_t}{X_t^2}}{\int_0^T\frac{1}{X_t^2}dt},\quad\text{ in probability, as $\dt\to0$}.\]

Further, from Lemma \ref{lem:convergence-of-QuoT}, we know that the random variable $\quo(T)$ converges to
$$
\quo_\infty:=\frac{\theta\sigma^2}{c^{1/p}\theta^{1-1/p}+\sigma^2}>0,\quad a.s.\text{ as }T\rightarrow\infty.
$$

Therefore
\begin{align*}
\P\left(\quo(T,\dt) \leq  0 \right) &= \P\left(\quo(T,\dt) \leq  0 , |\quo(T,\dt)-\quo(T)|>\quo_\infty/2\right) \\
&\quad + \P\left(\quo(T,\dt) \leq  0 , |\quo(T,\dt)-\quo(T)|\leq \quo_\infty/2\right)    \\
  & \leq \P\left(|\quo(T,\dt)-\quo(T)|>\quo_\infty/2\right)  + \P\left(\quo(T) \leq \quo_\infty/2\right),
\end{align*}
with both terms in the right hand side of the inequality converging to 0 when $\dt\to0$ and $T\to\infty.$
We conclude that, for any $k\geq1$ fixed:
$$
\lim_{T\to\infty}\lim_{\dt\to0}\P\left(\hat{\theta}^{(k,T,\dt)}   \leq  0 \right)  = 0.
$$

\end{proof}

\subsection{Convergence of Moment-EM estimator}

Before stating our main result, we make the following assumption:

\begin{assumption}\label{ass:data}
The dataset $\mathcal{D}$ corresponds to a direct observation of the solution of the model \eqref{eq:SM1}.
\end{assumption}

The convergence of the Moment-EM estimator is established in the asymptotic regime where the number of iterations and the observation horizon tend to infinity, while the mesh size of the observation grid tends to zero. To prove consistency, we consider these limits in an iterated way, for which we introduce the following surrogate estimators, serving as convenient analytical proxies for studying the asymptotic behavior of the estimator:

\begin{definition} \label{def:surrogate_definition}
Let $(X_{t})_{t\geq0}$ be the solution of the MVSDE \eqref{eq:SM1}. The continuously sampled estimator at iteration $k\ge1$, is defined as $\hat{\theta}^{(0,T)}=0$ and
\begin{equation}\label{eq:first-surrogate-estimator}
\begin{aligned}
    \hat{\theta}^{(k,T)}& = \theta  +  \frac{\sigma\int_0^T  \frac{1}{X_s}dW_s }{\int_0^T  \frac{1}{X_s^2}ds }+ \frac{c\int_{0}^T \frac{({u}_{s}^{(k,T)})^{p-1}-\uu_s^{p-1}}{X_s}ds}{\int_0^T  \frac{1}{X_s^2}ds},\\
    u_t^{(k,T)} &= \uu_0+ \int_0^t\Big(\hat{\theta}^{(k-1,T)}-c\left(u^{(k,T)}_s\right)^p\Big)ds.
\end{aligned}
\end{equation}
\end{definition}

\medskip

Notice the reminicense with the estimator proposed by Hu and Nualart in \cite{HU20101030}.
In the following propositions we establish the properties of $\hat{\theta}^{(k,T)}$ and $\hat{\theta}^{(k,T,\dt)}$. Their proofs are postponed to the Appendices \ref{ap:proofs-convergence-surrogate-estimators} and \ref{app:on-discrete-estimator}.

\begin{proposition}
\label{prop:thetaT-is-eventually-positive} There exists a positive and finite random variable $\TT$ and a constant $\cotaMin:=\cotaMin(\theta,c,\sigma,p)$, such that for all $k\geq1$ the continuously sampled estimator satisfies 
$$\hat{\theta}^{(k,T)}(\omega)\geq \cotaMin,\; \forall\,T\geq \TT(\omega).$$
\end{proposition}

\begin{proposition}
\label{e2:convergence-in-T} The following iterated limit holds almost surely:
$$
\lim_{k\to\infty}\lim_{T\to\infty }\hat{\theta}^{(k,T)} = \theta.
$$
Moreover,
$$
\sqrt{\frac{ T}{\varInfty}} \left(\hat{\theta}^{(k,T)} -\theta-\frac{c\int_{0}^T \frac{({u}_{s}^{(k,T)})^{p-1}-\uu_s^{p-1}}{X_s}ds}{\int_0^T  \frac{1}{X_s^2}ds}\right)\Rightarrow \mathcal{N}(0,1),\quad\text{ as $T\to\infty$,} 
$$
where
\begin{equation}\label{eq:def-Var-infty}
\varInfty = \frac{2\theta^2\sigma^2}{(2c^{1/p}\theta^{1-1/p}+\sigma^2)(c^{1/p}\theta^{1-1/p}+\sigma^2)}.
\end{equation}
\end{proposition}

\begin{proposition}
\label{e1:convergence-in-dt} It holds for any $\epsilon>0$:
$$
\lim_{\dt\to0}\P\left( \left| \hat{\theta}^{(k,T,\dt)}- \hat{\theta}^{(k,T)}\right| > \epsilon , T\geq \TT  \right) = 0.
$$
\end{proposition}
Now we proceed to state and proof the main result of this article:
\begin{theorem}[Consistency of the Moment-EM estimator]
%Let us assume that $\hat{\theta}^{(1,T,\dt)}>0$, then for all $\epsilon>0$,
Under Assumption \ref{hip:sign_theta},  for all $\epsilon>0$,
\begin{equation}\label{eq:main-convergence}
\lim_{k\to\infty}\lim_{T\to\infty}\lim_{\dt\to0} \P\left(\left|\hat{\theta}^{(k,T,\dt)}-\theta\right| > \epsilon \right) =0.
\end{equation}

\end{theorem}

\begin{proof}
Notice that
\begin{equation*}
\begin{aligned}
 \P\left(\left|\hat{\theta}^{(k,T,\dt)}-\theta\right| > \epsilon \right) &= \P\left(\left|\hat{\theta}^{(k,T,\dt)}-\theta\right| > \epsilon ,T<\TT \right) + \P\left(\left|\hat{\theta}^{(k,T,\dt)}-\theta\right| > \epsilon, T\geq \TT \right)\\
 &\leq \P\left(T<\TT \right) + \P\left(\left|\hat{\theta}^{(k,T,\dt)}-\hat{\theta}^{(k,T)}\right| > \epsilon/2, T\geq \TT \right) \\ 
 & \quad +\P\left(\left|\hat{\theta}^{(k,T)}-\theta\right| > \epsilon/2, T\geq \TT \right)\\
 &\leq \P\left(T<\TT \right) + \P\left(\left|\hat{\theta}^{(k,T,\dt)}-\hat{\theta}^{(k,T)}\right| > \epsilon/2, T\geq \TT \right) \\
  &\quad +\P\left(\left|\hat{\theta}^{(k,T)}-\theta\right| > \epsilon/2 \right).
 \end{aligned}
\end{equation*}
Then, thanks to Propositions \ref{e2:convergence-in-T} and  \ref{e1:convergence-in-dt} we have, 
\begin{align*}
\limsup_{k\to\infty}\limsup_{T\to\infty}\limsup_{\dt\to0} \P\left(\left|\hat{\theta}^{(k,T,\dt)}-\theta\right| > \epsilon \right) &\leq    \limsup_{T\to\infty}\P\left(T<\TT \right).
\end{align*}
We conclude applying Proposition \ref{prop:thetaT-is-eventually-positive}.
\end{proof}

\section{Numerical experiments}\label{sec:numerical_experiments}

In this section, we illustrate the performance of the proposed EM estimator through a series of numerical experiments under  different combinations of the parameters $p$, $\theta$, $\sigma$, $\dt$ and $T$. We generate synthetic data by simulating the McKean-Vlasov SDE \eqref{eq:SM1} using an exponential scheme with a discretization step of $\delta = 10^{-4}$. The values of $c$ is fixed at $c=1$ throughout this section, while $X_0 \sim \mathcal{U}(2,5)$, with $ u_0 = \E[X_0]$. We generate 5000 independent Monte Carlo sample paths to estimate all reported statistics.

In the computational implementation, unless explicitly indicated, we fix a tolerance level of $\Tol=10^{-5}$ and stop the algorithm  when the change in the estimator between two successive iterations falls below this threshold. We denote by $k^*$ the index of the last iteration, i.e., 
$$k^* := \min\{k \geq 1 : |\hat{\theta}^{(k,T,\Delta t)} - \hat{\theta}^{(k-1,T,\Delta t)}| < \Tol\},$$
 and refer to $\hat{\theta}^{(k^*,T,\Delta t)}$ as the final estimator. The performance of the estimator is measured through the  relative error $|\hat{\theta}^{(k^*,T,\Delta t)} - \theta|/\theta$, and its empirical mean across replications, which we refer to as the Mean Relative Error (MRE).

\paragraph{The role of $p,\,\theta$ and $T$:}
In Figure \ref{Fig:convergence}, we show the evolution of the MRE and its $[20\%,80\%]$ confidence interval across iterations.  In all cases, $k^*\in[10,30]$, with $k^*$ increasing with the values of $p$ and $\theta$ (see Remark \ref{rem:geometric-convergence-theta-k} in the Appendix for some additional comments). Notice also that the width of the confidence interval, represented by the shaded regions, increases with $p$, and decreases with $\theta$ and  the time horizon $T$. Finally, in all cases, for $T=100$ we obtain the narrowest empirical confidence interval and the smallest $k^*$. As complement for Fig. \ref{Fig:convergence}, in  Table \ref{Tab:Table1} we report the empirical mean and standard deviation of the relative error.

\begin{figure}[htbp]
    \centering
    \begin{tabular}{m{0.2cm} ccc}
        
        % --- ENCABEZADOS DE COLUMNA ---
        & \textbf{$p = 1.5$} & \textbf{$p = 2.0$} & \textbf{$p = 2.5$} \\ [1.5ex]

        % --- FILA 1 ---
        \rotatebox{90}{\textbf{$\theta = 0.1$}} &
        \begin{subfigure}[b]{0.29\textwidth}
            \adjincludegraphics[width=0.9\textwidth,valign=m]{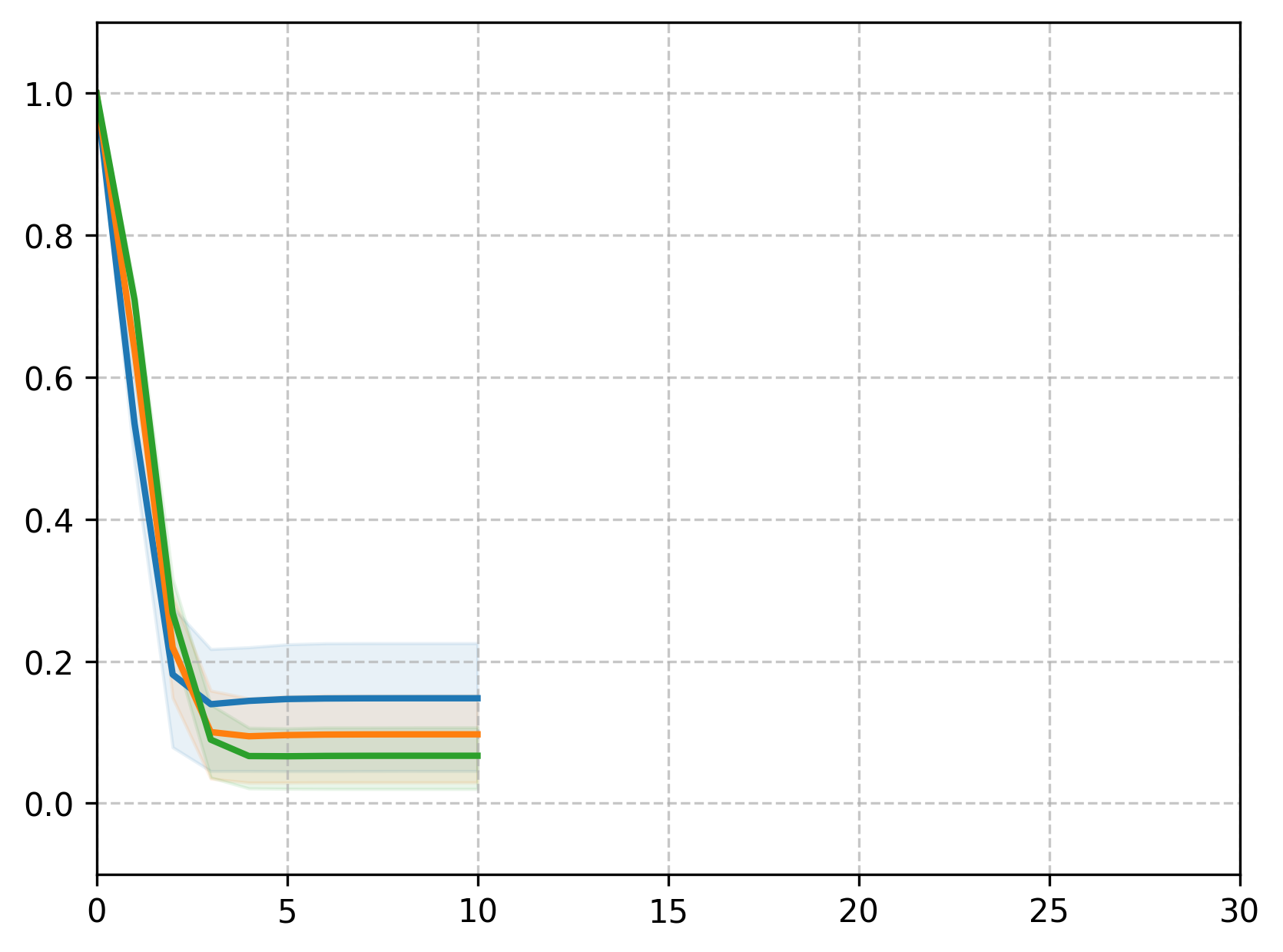}
        \end{subfigure} &
        \begin{subfigure}[b]{0.29\textwidth}
            \adjincludegraphics[width=0.9\textwidth,valign=m]{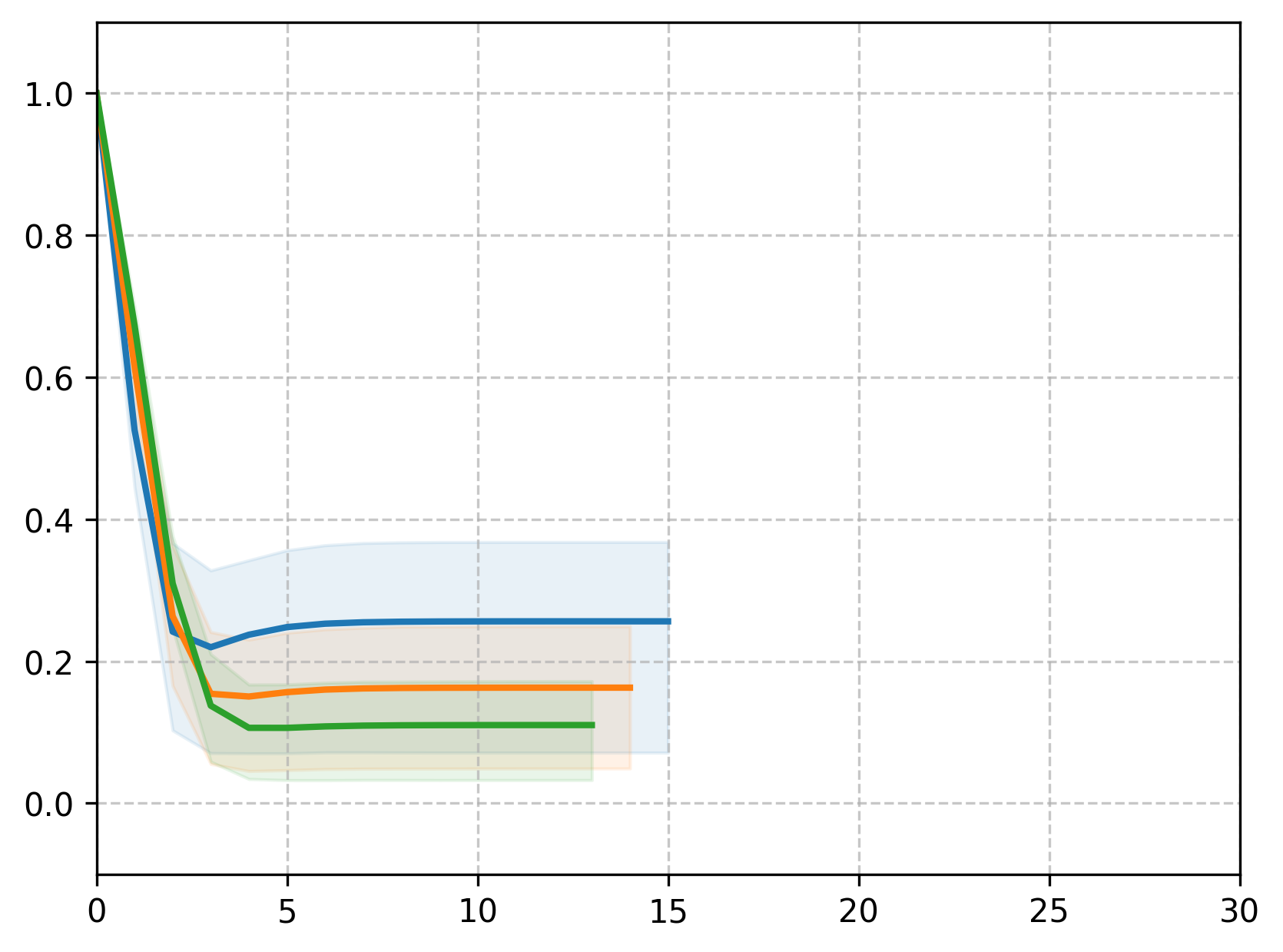}
        \end{subfigure} &
        \begin{subfigure}[b]{0.29\textwidth}
            \adjincludegraphics[width=0.9\textwidth,valign=m]{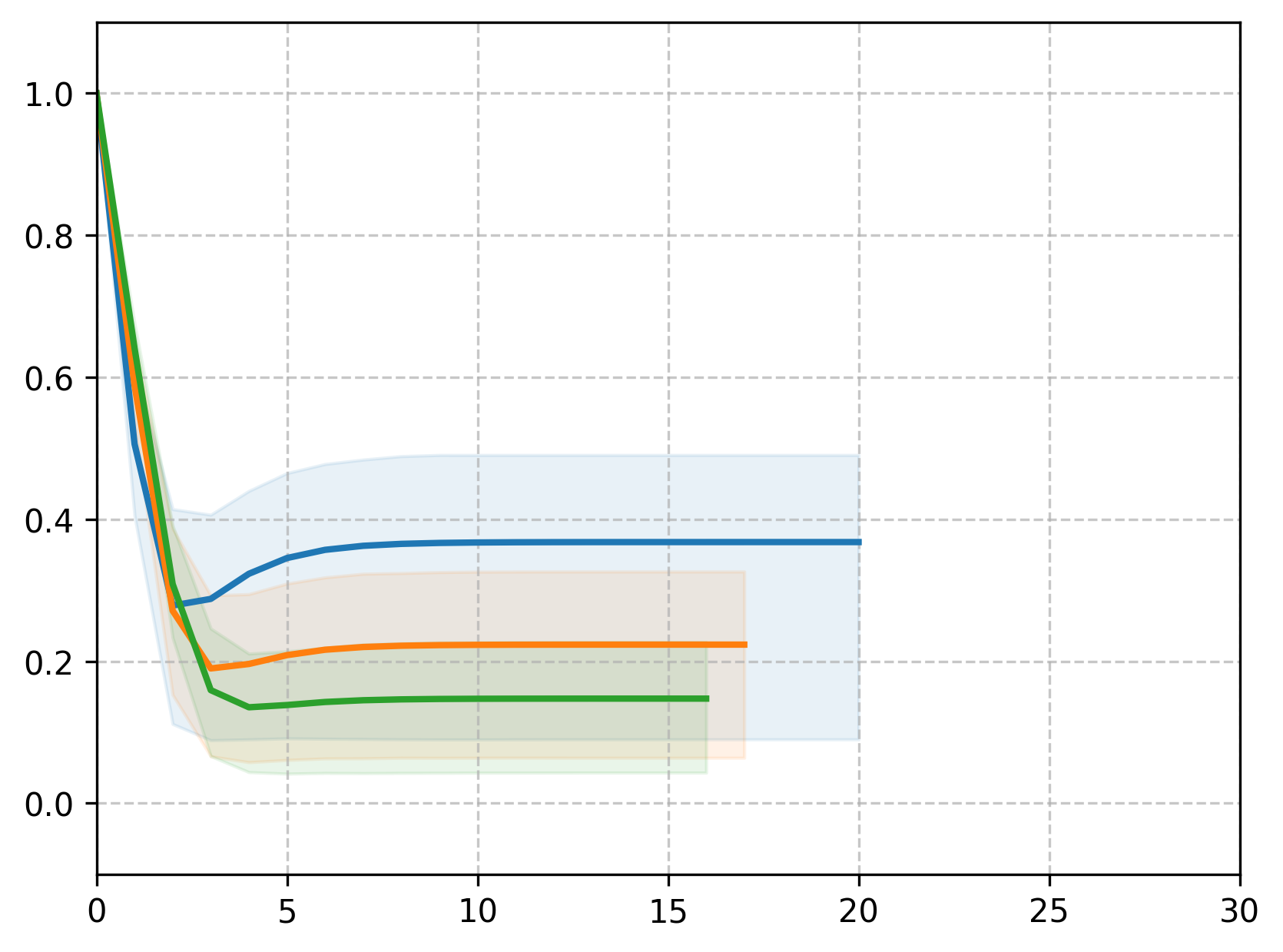}
        \end{subfigure} \\ [3ex]

        % --- FILA 2 ---
        \rotatebox{90}{\textbf{$\theta = 2.0$}} &
        \begin{subfigure}[b]{0.29\textwidth}
            \adjincludegraphics[width=0.9\textwidth,valign=m]{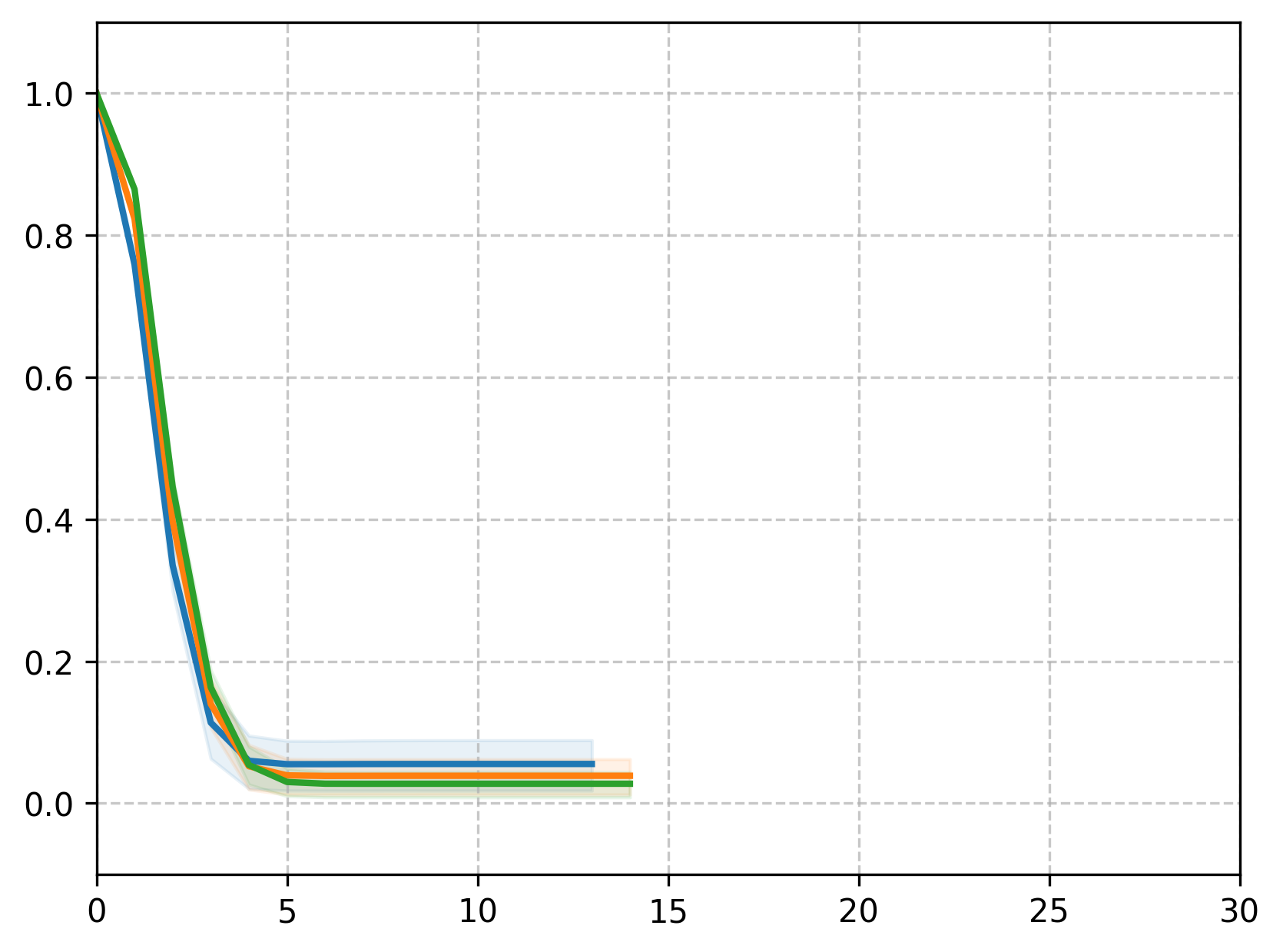}
        \end{subfigure} &
        \begin{subfigure}[b]{0.29\textwidth}
            \adjincludegraphics[width=0.9\textwidth,valign=m]{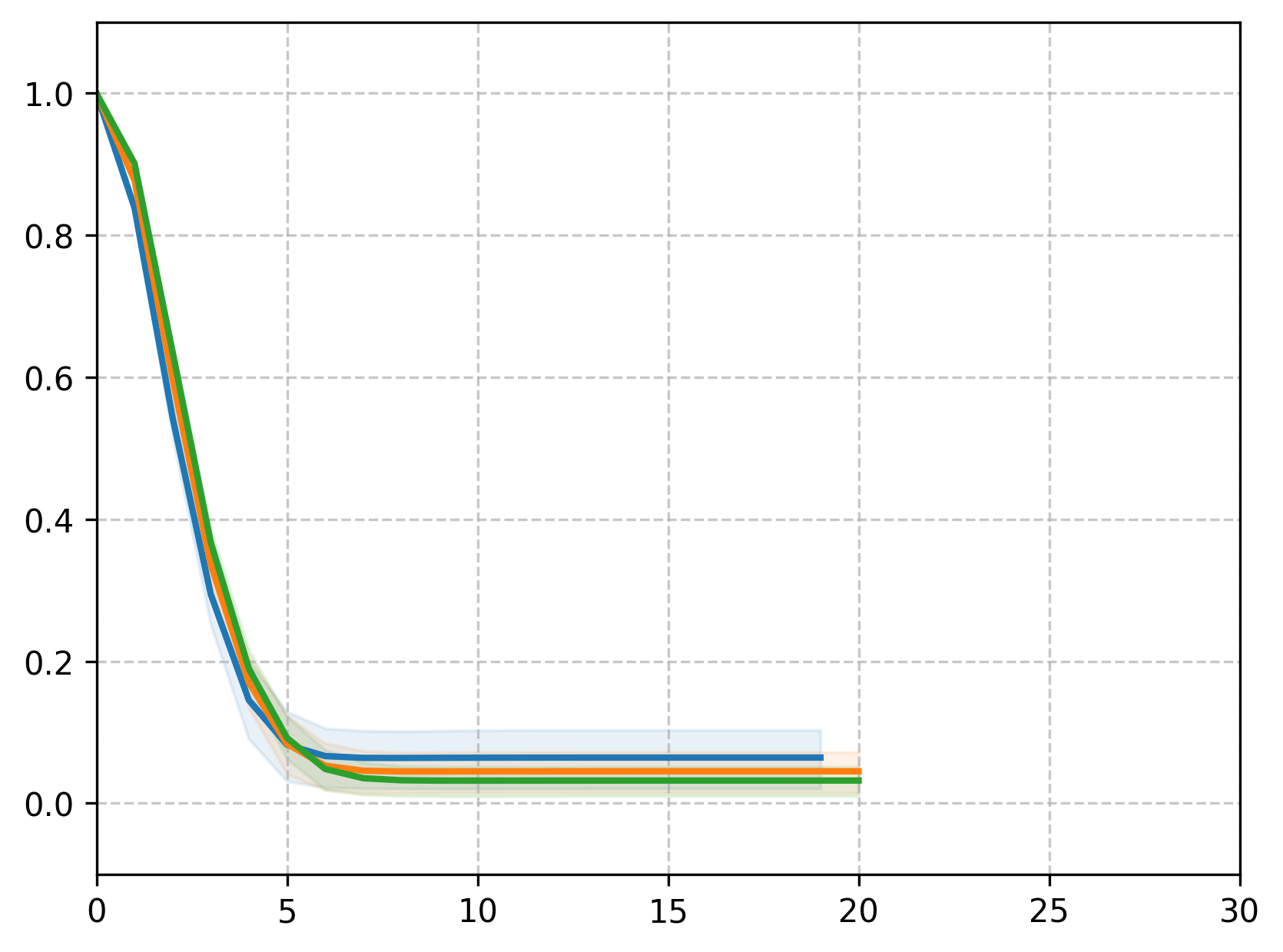}
        \end{subfigure} &
        \begin{subfigure}[b]{0.29\textwidth}
            \adjincludegraphics[width=0.9\textwidth,valign=m]{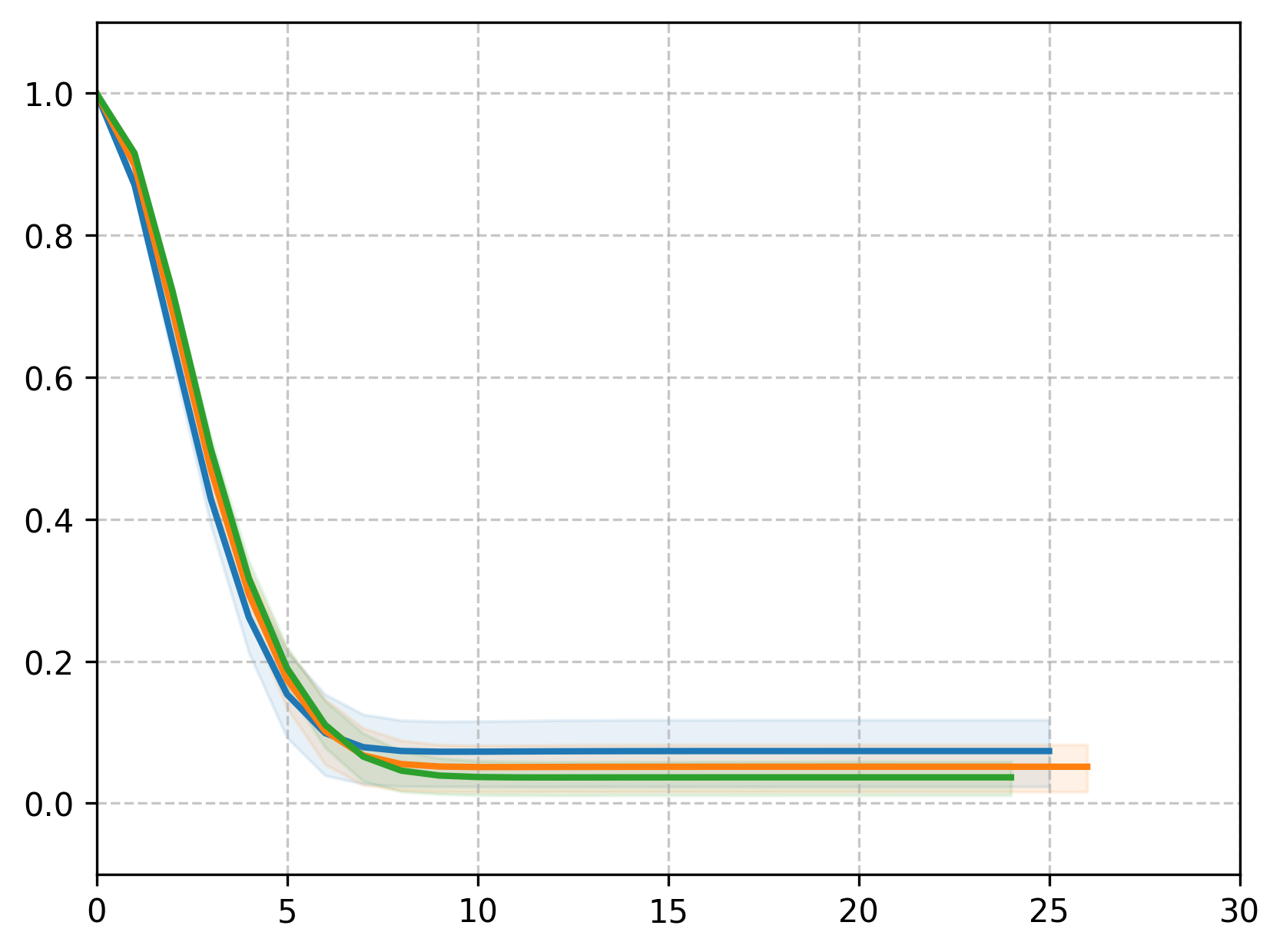}
        \end{subfigure} \\ [3ex]

        % --- FILA 3 ---
        \rotatebox{90}{\textbf{$\theta = 4.0$}} &
        \begin{subfigure}[b]{0.29\textwidth}
            \adjincludegraphics[width=0.9\textwidth,valign=m]{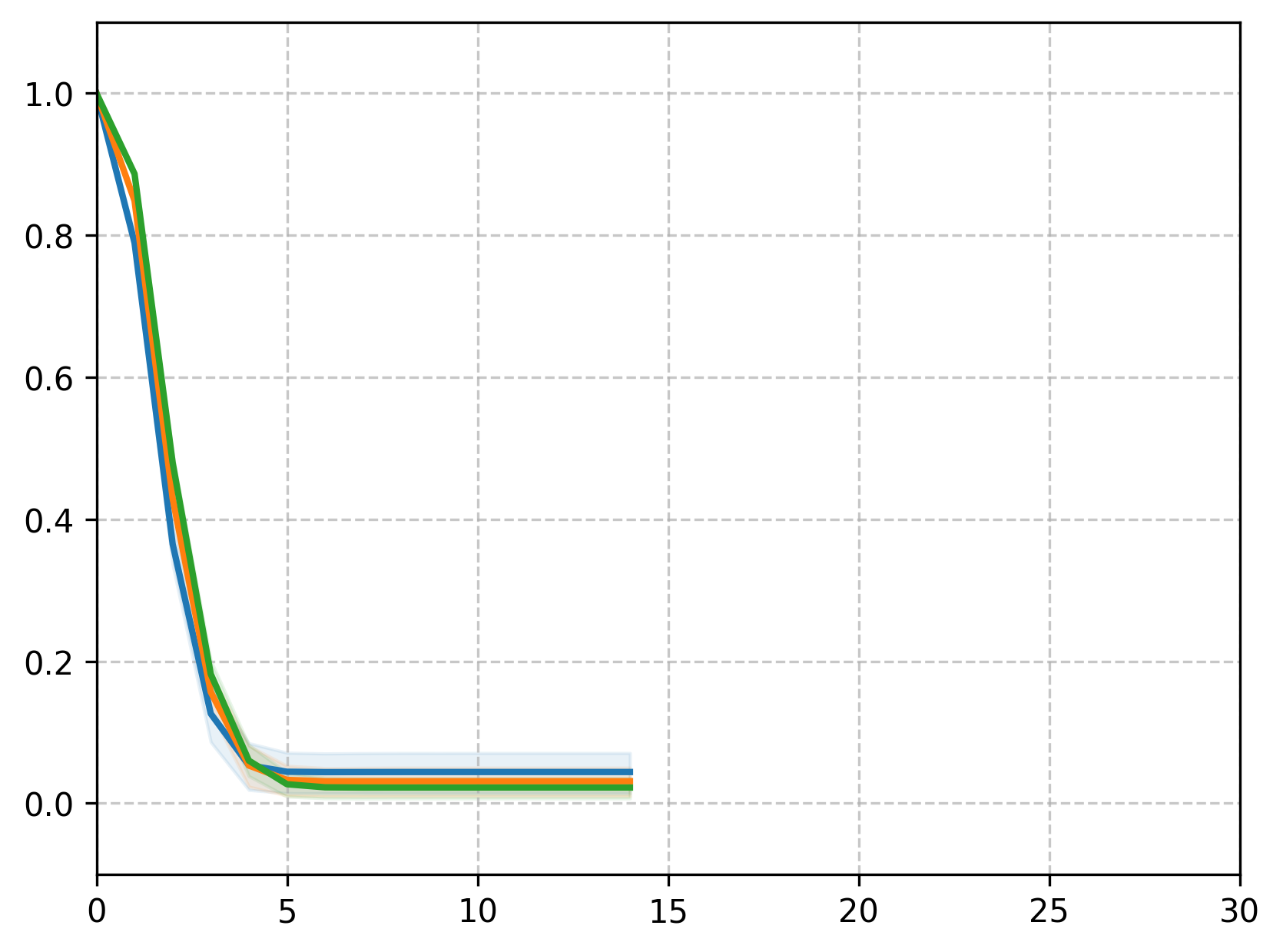}
            
        \end{subfigure} &
        \begin{subfigure}[b]{0.29\textwidth}
            \adjincludegraphics[width=0.9\textwidth,valign=m]{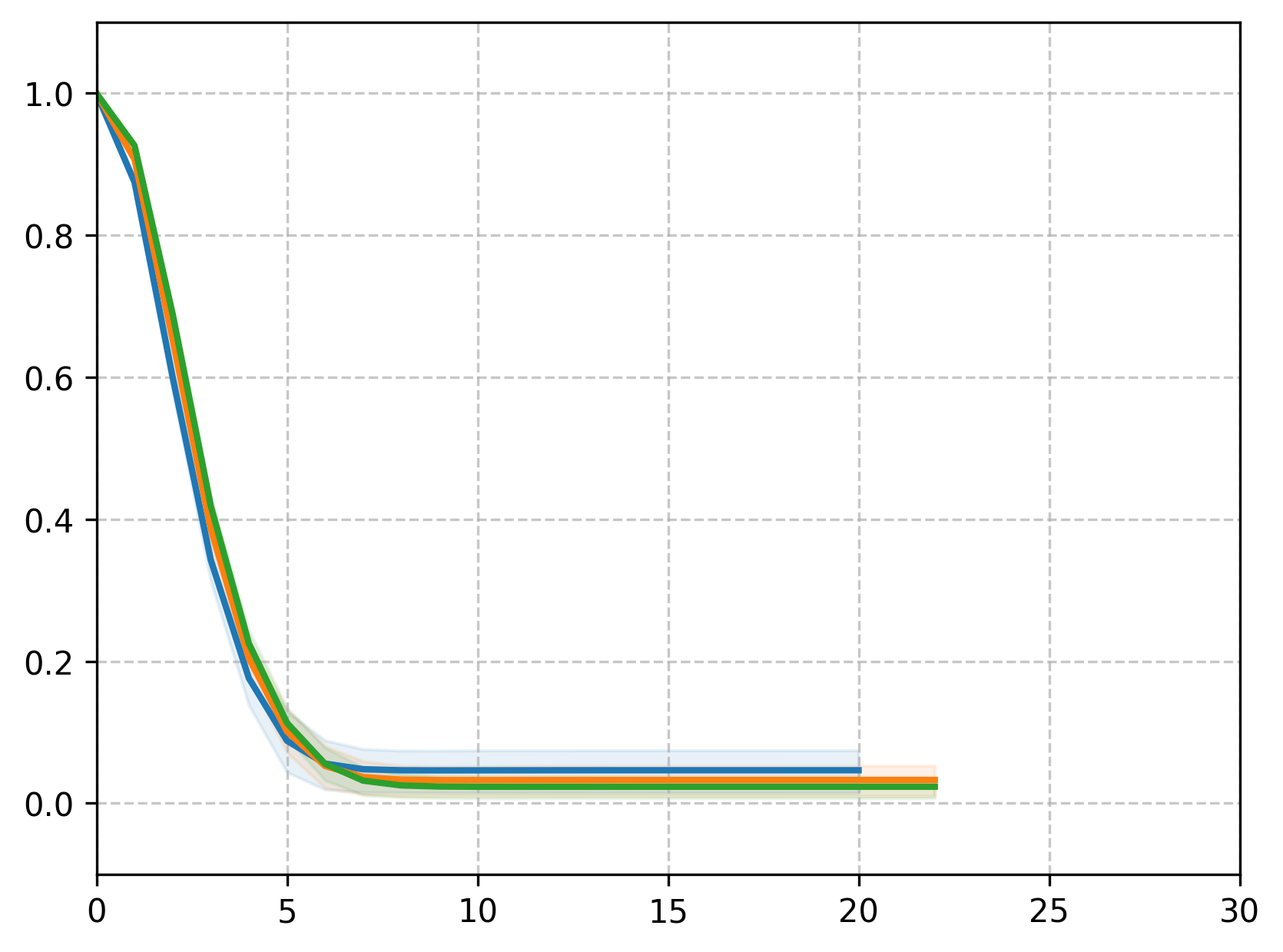}
            
        \end{subfigure} &
        \begin{subfigure}[b]{0.29\textwidth}
            \adjincludegraphics[width=0.9\textwidth,valign=m]{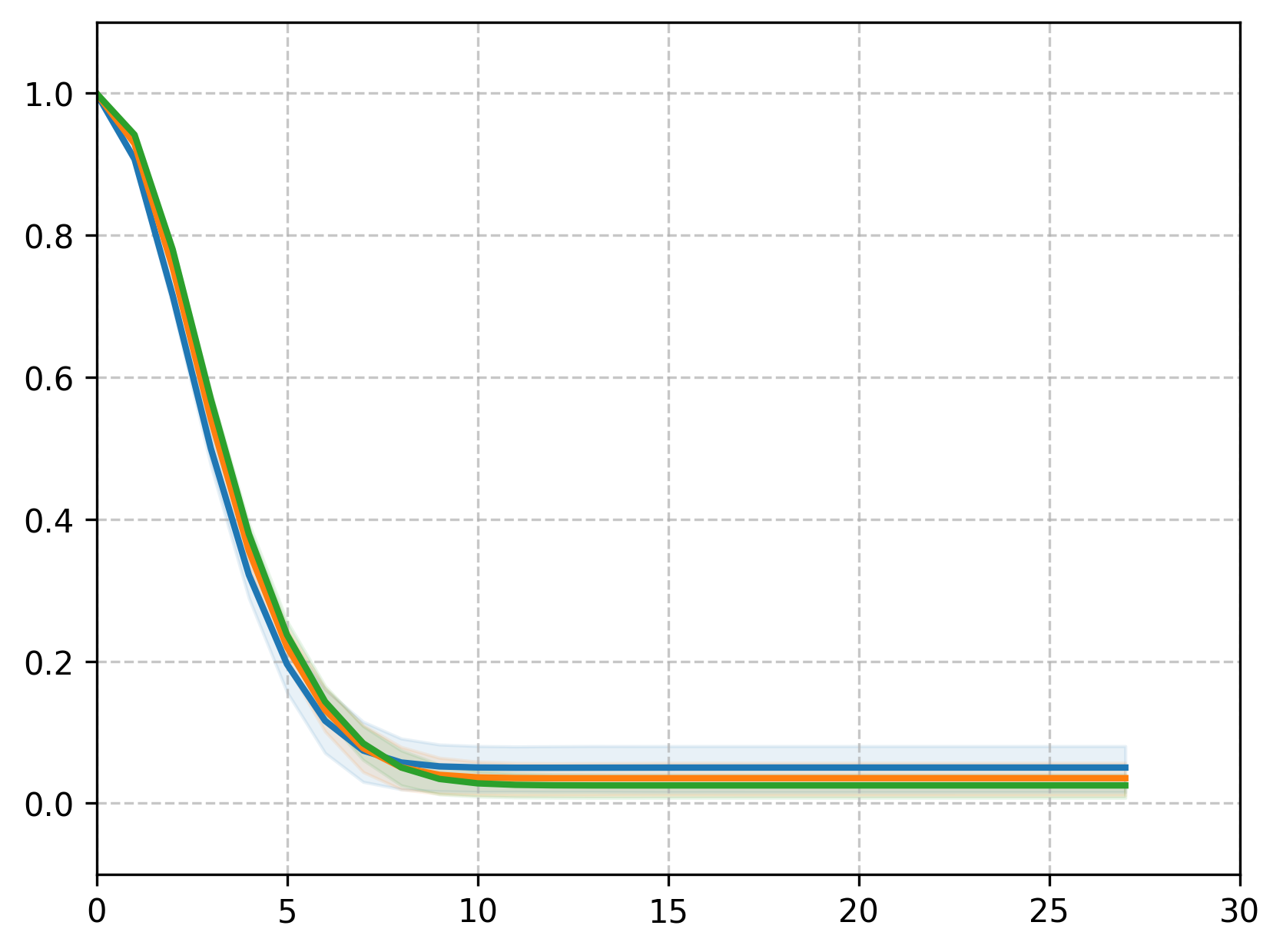}

        \end{subfigure}
    \end{tabular}
    \captionsetup{width=0.9\textwidth}
    \caption{Evolution of the distribution of the relative error across iterations $k=1,\ldots,k^*$. Solid lines denote the empirical mean $T=25$ (blue), $T=50$ (orange), and $T=100$ (green). Shaded regions represent the empirical $(20\%,80\%)$-confidence interval. Fixed parameters: $\sigma^2=0.1$, $\Delta t=10^{-4}$.}
    \label{Fig:convergence}
\end{figure}

\begin{table}[h!]
    \centering
    \small
\begin{tabular}{l cc cc cc}
        \toprule
        & \multicolumn{6}{c}{\textbf{Parameter $p$}} \\
        \cmidrule(lr){2-7}
        & \multicolumn{2}{c}{$p = 1.5$} & \multicolumn{2}{c}{$p = 2.0$} & \multicolumn{2}{c}{$p = 2.5$} \\
        \cmidrule(lr){2-3} \cmidrule(lr){4-5} \cmidrule(lr){6-7}
        \textbf{Parameter $\theta$} & Mean & SD & Mean & SD & Mean & SD \\
        \midrule
        $\theta = 0.1$  & 0.0669 & 0.0536 & 0.1102 & 0.0933 & 0.1474 & 0.1330 \\
        $\theta = 2.0$  & 0.0275 & 0.0212 & 0.0320 & 0.0246 & 0.0365 & 0.0282 \\
        $\theta = 4.0$  & 0.0222 & 0.0170 & 0.0233 & 0.0179 & 0.0251 & 0.0193 \\
        $\theta = 8.0$  & 0.0178 & 0.0136 & 0.0168 & 0.0129 & 0.0171 & 0.0131 \\
        $\theta = 10.0$ & 0.0165 & 0.0127 & 0.0151 & 0.0116 & 0.0151 & 0.0115 \\
        $\theta = 20.0$ & 0.0132 & 0.0101 & 0.0109 & 0.0083 & 0.0101 & 0.0077 \\
        \bottomrule
    \end{tabular}
    \captionsetup{width=0.9\textwidth}
    \caption{Empirical mean and standard deviation of the relative error for varying $\theta$ and $p$. Fixed parameters: $\sigma^2=0.1$, $\Delta t=10^{-4}$, and $T=100$.}
    \label{Tab:Table1}
\end{table}

\clearpage

\paragraph*{The role of diffusion parameter $\sigma^2$ and the sampling frequency $\Delta t$:}
Table \ref{Tab:Table2} reports the empirical mean and standard deviation of the relative error for varying $p$ and $\sigma^2$. Observe that both empirical statistics decrease monotonically with $\sigma^2$, no matter the value of $p$.
  
Figure \ref{Fig:MSE} shows the empirical MRE as a function of $\Delta t$. As expected, the MRE decreases as $\dt$ decreases, although a plateau appears due to bias of the estimator. This bias seems to depend on $\sigma^2$. More precisely, when $\sigma^2$ increases the bias seems to decrease. Heuristically, from Proposition \ref{e2:convergence-in-T}, consistency of the Moment-EM estimator, and Lemma \ref{lemma:integrals-dynamics-vs-equilbrium}, we have for sufficiently small $\dt$ and large enough $T$:
$$
\sqrt{\frac{T}{\varInfty} }\left( \hat{\theta}^{(k,T,\dt)}-\hat{\theta}^{(k)} \right)\approx \mathcal{N}(0,1),
$$
where
$$
\hat{\theta}^{(k)} = \frac{\sigma^2+c^{1/p}(\hat\theta^{(k-1)})^{1-1/p}}{\sigma^2+c^{1/p}\theta^{1-1/p}}\theta ,\;\;
\varInfty =\frac{2\theta^2\sigma^2}{(2c^{1/p}\theta^{1-1/p}+\sigma^2)(c^{1/p}\theta^{1-1/p}+\sigma^2)}.
$$
Equivalently, for sufficiently small $\dt$ and large enough $T$:
$$
\hat{\theta}^{(k,T,\dt)}\approx \mathcal{N}\left(\hat{\theta}^{(k)},\quad \frac{2\theta^2\sigma^2}{T\, (2c^{1/p}\theta^{1-1/p}+\sigma^2)(c^{1/p}\theta^{1-1/p}+\sigma^2)}\right).
$$
Notice that in the limit $\sigma^2\to\infty$, $\hat{\theta}^{(k)} = \theta$ and there is no bias. Meanwhile, in the limit $\sigma^2\to0$, $\hat{\theta}^{(k)}=0$, and the finite size bias is equal to the parameter.  Additionally, we have observed in the numerical experiments that the empirical variance of the estimator is consistent with the asymptotic variance $\varInfty$. 

\begin{table}[h]
    \centering
    \small
    \begin{tabular}{l cc cc cc}
        \toprule
        & \multicolumn{6}{c}{\textbf{Parameter $p$}} \\
        \cmidrule(lr){2-7}
        & \multicolumn{2}{c}{$p = 1.5$} & \multicolumn{2}{c}{$p = 2.0$} & \multicolumn{2}{c}{$p = 2.5$} \\
        \cmidrule(lr){2-3} \cmidrule(lr){4-5} \cmidrule(lr){6-7}
        \textbf{Parameter $\sigma^2$} & Mean & SD & Mean & SD & Mean & SD \\
        \midrule
        $\sigma^2 = 2.0$   & 0.0486 & 0.0377 & 0.0504 & 0.0391 & 0.0521 & 0.0404 \\
        $\sigma^2 = 4.0$   & 0.0426 & 0.0326 & 0.0435 & 0.0332 & 0.0442 & 0.0338 \\
        $\sigma^2 = 6.0$   & 0.0380 & 0.0288 & 0.0384 & 0.0292 & 0.0388 & 0.0295 \\
        $\sigma^2 = 8.0$   & 0.0345 & 0.0260 & 0.0348 & 0.0263 & 0.0350 & 0.0265 \\
        $\sigma^2 = 10.0$  & 0.0317 & 0.0239 & 0.0319 & 0.0241 & 0.0321 & 0.0242 \\
        $\sigma^2 = 25.0$  & 0.0216 & 0.0162 & 0.0216 & 0.0163 & 0.0217 & 0.0163 \\
        $\sigma^2 = 50.0$  & 0.0157 & 0.0118 & 0.0158 & 0.0118 & 0.0158 & 0.0118 \\
        $\sigma^2 = 100.0$ & 0.0114 & 0.0086 & 0.0114 & 0.0086 & 0.0114 & 0.0086 \\
        \bottomrule
    \end{tabular}
    \captionsetup{width=0.9\textwidth}
    \caption{Empirical mean and standard deviation of $|\hat{\theta}^{(k^*, T,\Delta t)}-\theta|/\theta$ for varying $p$ and $\sigma^2$. Fixed parameters:  $\theta=2$,  $\Delta t =10^{-4}$, and $T=100$.}
    \label{Tab:Table2}
\end{table}

\begin{figure}[h]
    \centering
    \includegraphics[scale=0.41]{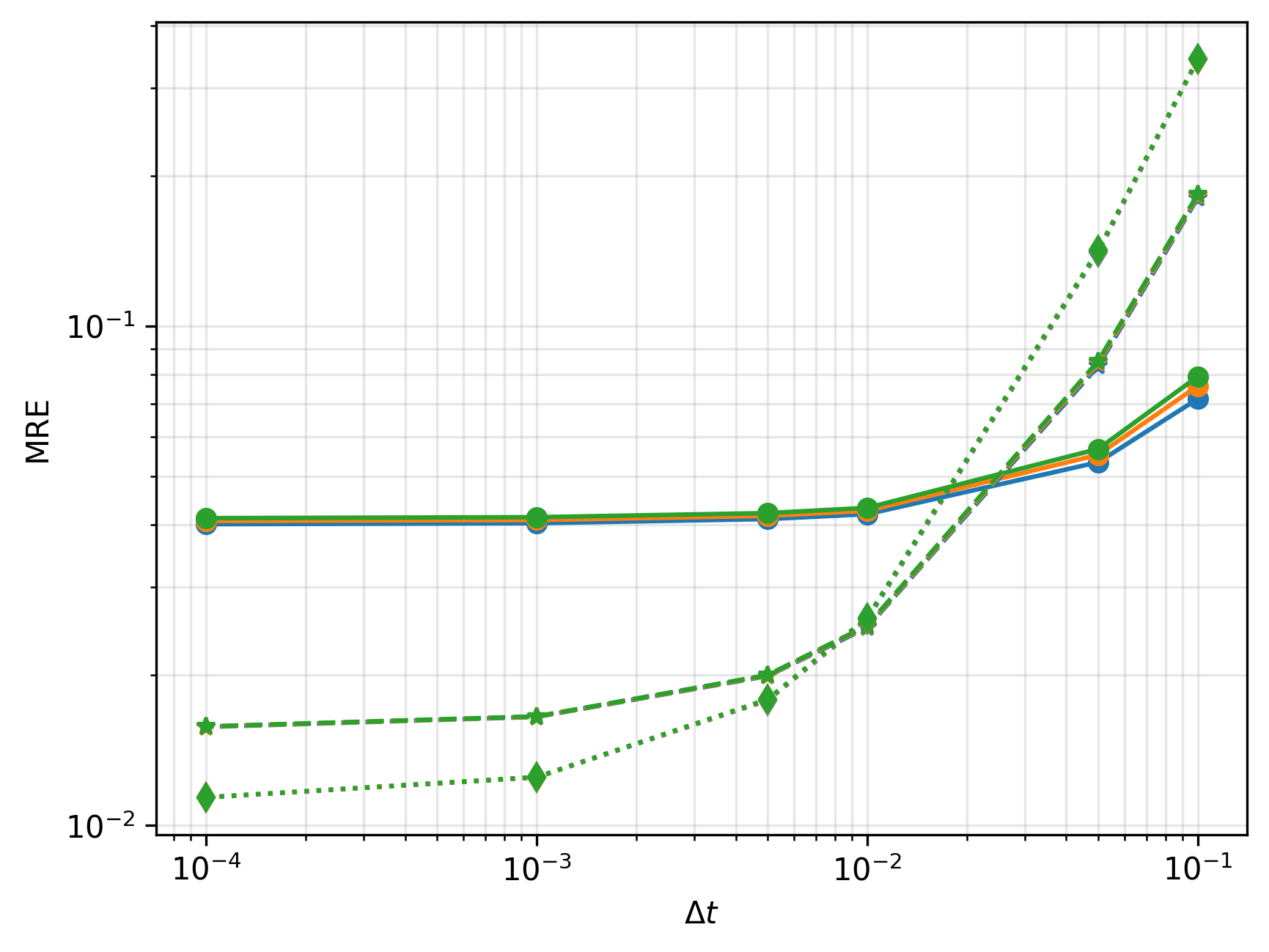}
    \captionsetup{width=0.9\textwidth}
    \caption{Empirical MRE for $\hat{\theta}^{(k^*,T,\dt)}$ vs. $\dt$ (log-log scale) for different values of $\sigma^2$ and $p$: $\sigma^2 = 5$ (solid), $\sigma^2 = 50$ (dashed), $\sigma^2 = 100$ (doted), $p =1.5$ (blue), $p=2.0$ (orange), and $p=2.5$ (green). Fixed parameters:  $\theta =2$, $T=100$, $\Delta t = 10^{-4} $ and $\Tol = 10^{-6}$.}
    \label{Fig:MSE}
\end{figure}

\clearpage
\paragraph*{The role of the regularity of the sampling grid:}

To assess the robustness of the method with respect to the regularity of the sampling grid, we implemented two sampling procedures to simulate changes in observation frequency. First, over a horizon $T = 100$, we perform a two-regime sampling scheme: the process is observed at $\Delta t_a = 0.1$ up to a switching time $t_{\text{switch}}$, and thereafter, we sample with $\Delta t_b = 0.001$; we also consider the reverse order (high-to-low frequency). For the switching time, we study two cases: $t_{\text{switch}} = 25$ and $t_{\text{switch}} =75$, while the initial condition is given by $X_0 \sim \mathcal{U}(0.9\Eq,1.1\Eq)$. Figure~\ref{fig:Samlping-Freq_a} shows the empirical distribution of the estimator for all four combinations, together with the two homogeneous references. We observe that curves with the same amount of time spent in each regime coincide almost exactly, regardless of which frequency comes first, suggesting that the order of the regimes has no visible effect on the estimator's distribution. What matters is the fraction of the horizon spent at each sampling rate, with the estimator's distribution closely tracking the homogeneous reference of whichever regime occupies the larger share of $[0,T]$. Still, the four distributions remain reasonably close to one another, suggesting that longer exposure to coarser sampling does not strongly degrade the estimator.

Second, we consider sampling at random times following a homogeneous Poisson Process with intensity $\lambda = 100$. In Figure \ref{fig:random-Samlping}, we observe that there is no qualitative difference in the empirical distribution of the algorithm with random sampling and the algorithm with constant frequency equal to $1/100$. This indicates that the algorithm is robust with respect to moderate perturbations of the sampling times.  

\begin{figure}[h]
    \centering
    \begin{subfigure}[b]{0.48\textwidth}
        \adjincludegraphics[width=\textwidth]{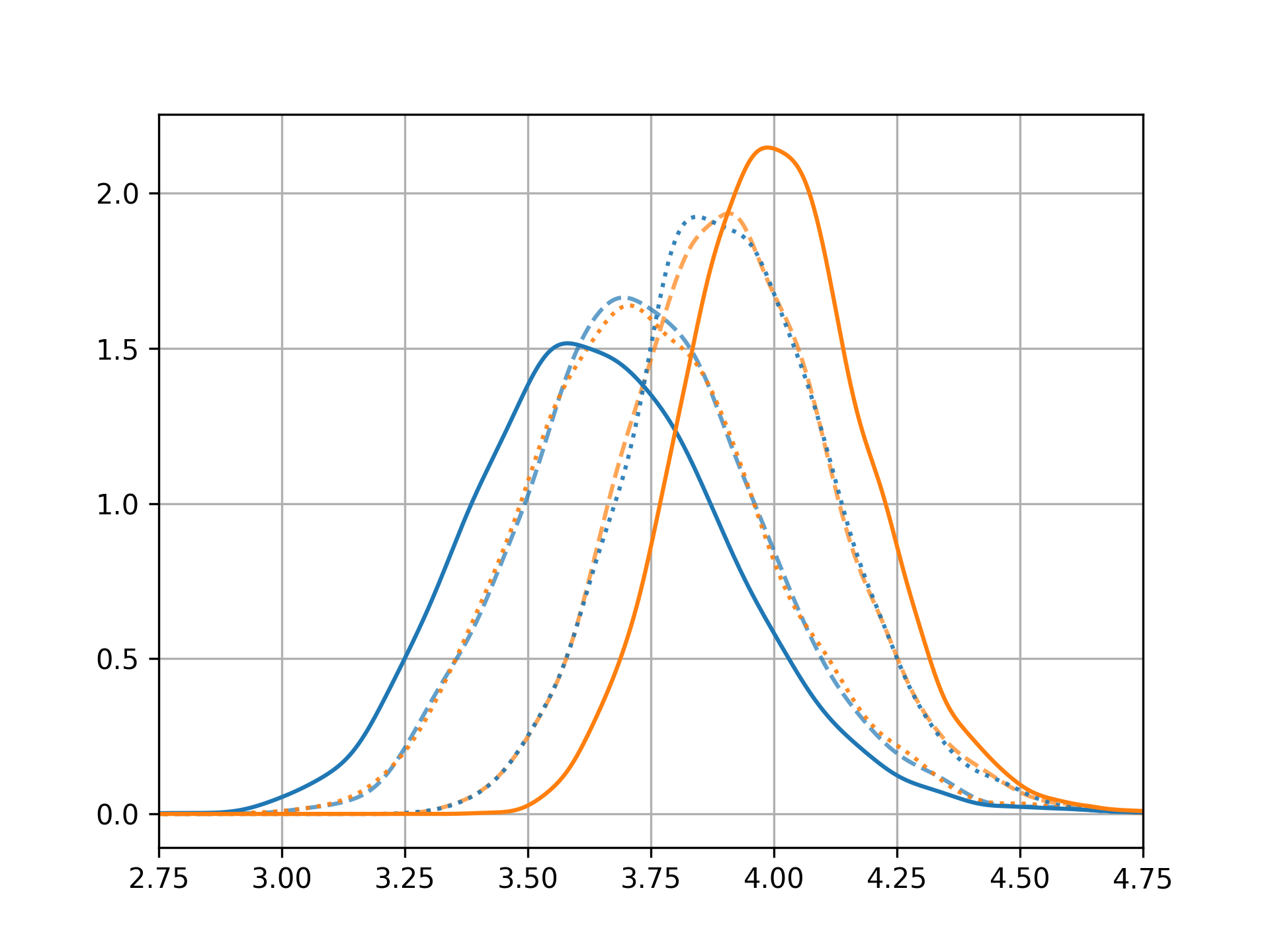}
        \captionsetup{width=0.9\textwidth}
        \caption{Homogeneous vs. switching sampling.}
        \label{fig:Samlping-Freq_a}
    \end{subfigure}
    \hspace{0.01\textwidth}
    \begin{subfigure}[b]{0.48\textwidth}
       
        \adjincludegraphics[width=.93\textwidth]{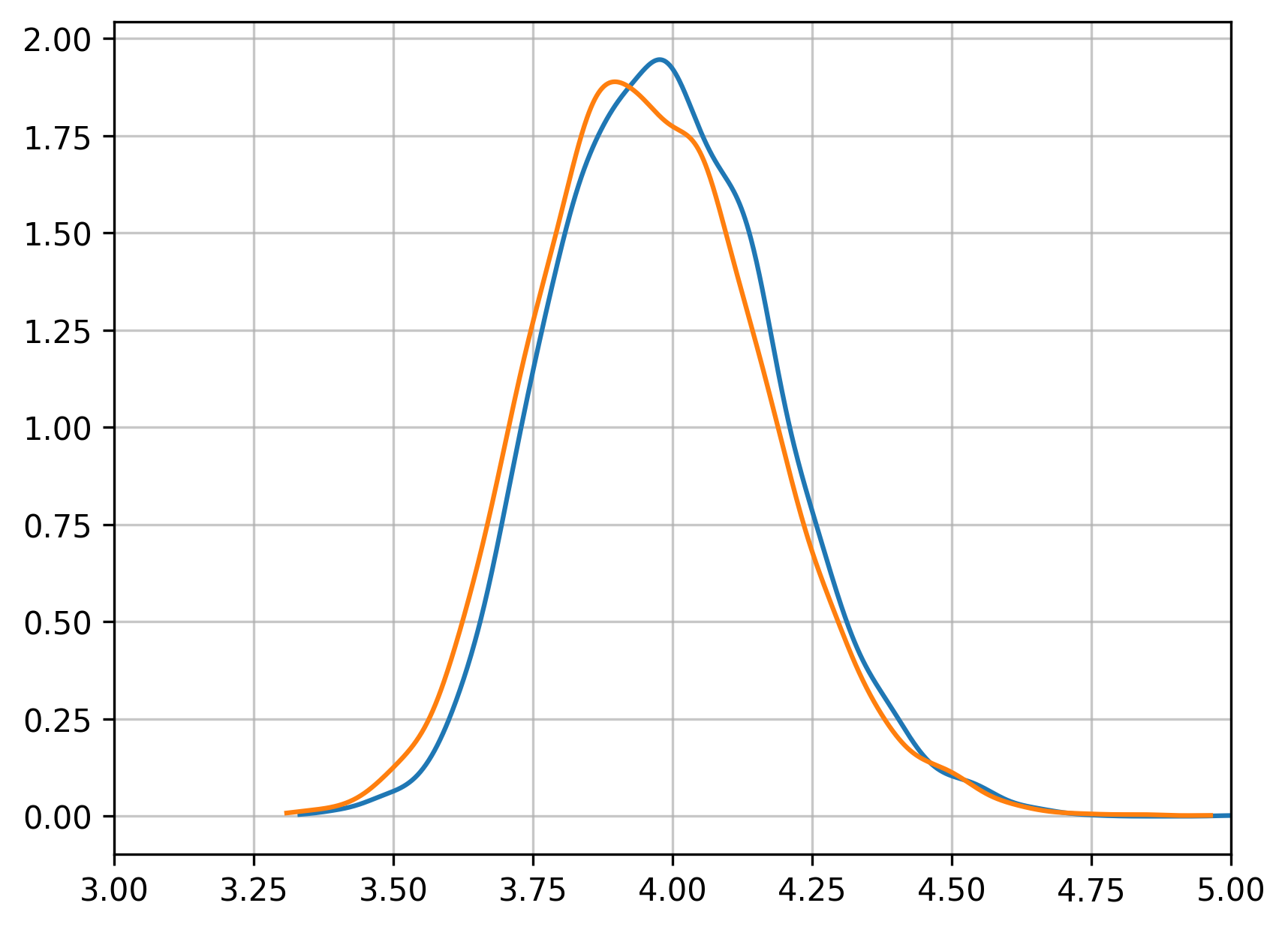}
        \captionsetup{width=0.9\textwidth}
        \caption{Homogeneous vs. Poissonian Sampling.}
         \label{fig:random-Samlping}
    \end{subfigure}
    \captionsetup{width=0.9\textwidth}

    \caption{Empirical PDF of $\hat{\theta}^{(k^*,T,\Delta t)}$ under non-regular sampling.
    {\bf (Left):} Two-regime switching sampling, compared to homogeneous references. Solid blue: homogeneous sampling with $\Delta t=0.1$; solid orange: homogeneous sampling with $\Delta t=0.001$. Blue (dotted/dashed): sampling starts at $\Delta t_a=0.1$ and switches to $\Delta t_b=0.001$. Orange (dotted/dashed): sampling starts at $\Delta t_b=0.001$ and switches to $\Delta t_a=0.1$. Dotted: switching time $t_{\text{switch}}=25$; dashed: switching time $t_{\text{switch}}=75$. Initial condition $X_0 \sim \mathcal{U}(0.9\Eq,1.1\Eq)$.
    {\bf (Right):} Random Poisson Sampling with intensity $100$ (orange), as reference homogeneous sampling with constant step $\Delta t=0.01$ (blue).
    Fixed parameters: $\theta=4$, $\sigma^2=5$, $T=100$.}
    \label{Fig:Irregular}
\end{figure}

\paragraph{The role of the tolerance of the algorithm:}
In Table \ref{Tab:Iter}, we provide descriptive statistics for $k^*$ across different tolerance levels. The results indicate, as expected, that the mean number of iterations necessary for convergence increases as the tolerance decreases. Nevertheless, we observe that the MRE stabilizes from tolerance $10^{-3}$, suggesting that there is no gain in considering smaller tolerances.

.

\begin{table}[h]
    \centering
    \small 
    \begin{tabular}{l cc cc cc}
        \toprule
        \multirow{2}{*}{$\text{Tol}$} & \multicolumn{2}{c}{\textbf{Mean $k^*$}} & \multicolumn{2}{c}{\textbf{Max}} & \multicolumn{2}{c}{\textbf{MRE}} \\
        \cmidrule(lr){2-3} \cmidrule(lr){4-5} \cmidrule(lr){6-7}
        & $\theta=0.1$ & $\theta=2.0$ & $\theta=0.1$ & $\theta=2.0$ & $\theta=0.1$ & $\theta=2.0$ \\
        \midrule
        $10^{-1}$ & 2.0000 & 6.0000  & 2 & 6 & 0.708866 & 0.030021 \\
        $10^{-2}$ & 5.0000 & 8.0000 &  5 & 8 & 0.066505 & 0.027468 \\
        $10^{-3}$ & 7.0000 & 10.0000 &  7 & 10 & 0.066698 & 0.027511 \\
        $10^{-4}$ & 8.9580 & 11.9998 &  9 & 12 & 0.066932 & 0.027518 \\
        $10^{-5}$ & 10.3516 & 13.9994  & 11 & 15 & 0.066945 & 0.027508 \\
        \bottomrule
    \end{tabular}
    \captionsetup{width=0.9\textwidth}
    \caption{$k^*$ and MRE of $\hat{\theta}^{(k^*,T,\Delta t)}$ for different tolerance levels and values of $\theta$. Max: maximum. Fixed parameters: $T=100$, $\Delta t=10^{-4}$, $p=1.5$, and $\sigma^2=0.1$.}
    \label{Tab:Iter}
\end{table}

\section{Closing remarks}\label{sec:closing}
In this article we have introduced an Expectation-Maximization type algorithm to estimate the forcing parameter of the MV model \eqref{eq:SM1}. The algorithm is easy to implement and its computational cost grows linearly on the size of the data. Moreover, we have proved analytically its consistency and established the assymptotic normality 
for the continuously-sampled estimator $\hat{\theta}^{(k,T)}$, with explicit asymptotic variance. The latter is consistent with  the empirical variance observed in the numerical experiments.
The numerical experiments also show that the time horizon $T$ is the most influential parameter on the performance of the estimator, implying that in practical terms, longer observational records are more valuable than increasing the sampling frequency---a finding analytically supported by Proposition \ref{e2:convergence-in-T} and consistent with the precedent set by Fournié and Talay \cite{FournieTalay1991}, who studied maximum likelihood and method-of-moments estimators for the Cox--Ingersoll--Ross (CIR) model and found that, while the volatility parameter can be accurately estimated via quadratic variation, precise estimation of the drift parameters requires long observation horizons. We also find evidence of robustness with respect to the sampling frequency and the power $p$.

Finally, the results of this work open several research directions. On the theoretical side, given the empirical 
evidence presented, a rigorous Central Limit Theorem for the discretely-sampled estimator $\hat{\theta}^{(k,T,\Delta t)}$ is a natural next step. On the applied side, extensions to models in higher dimensions and to more general diffusion coefficients are part of ongoing research.

%%%%%%%%%%%%%%%%%%%%%%%%%%%%%%%%%%%%%%%%%%
%%%%%%%%%%%%%%%%%%%%%%%%%%%%%%%%%%%%%%%%%%
\appendix
\section{Appendix}\label{sec:appendix}

\subsection{On the ODE }\label{app:preliminaries_proofs}
Formally, assuming there is a weak solution to Equation \eqref{eq:SM1} and taking expectations, we couple the model with the following initial value problem for the deterministic function $\uu_t:=\E[X_t]$:

\begin{equation}\label{eq:ODE}
\left\{ \begin{array}{lcc}
\dfrac{d\uu_t}{dt} = \theta-c\uu_t^{p} \\  
\\ 
\uu_0=\E[X_0].
\end{array} \right.
\end{equation}

When the dependence of $\uu$ on $\theta$ needs to be emphasized, we denote it by $\uu^{(\theta)}$.

\medskip
The properties of the solution to \eqref{eq:ODE} are crucial for constructing the EM-estimator and proving its consistency. We summarize them in the following lemma:

\begin{lemma}
\label{prop:ode_properties} Consider a fixed $\theta>0$ and assume $u_0>0$. Then, it holds:
\begin{enumerate}[label=(\roman*)]
\item\label{item:wellposedness_ODE} For any time horizon $T>0$, there exists a unique positive solution $\uu^{(\theta)}$ to the initial value problem \eqref{eq:ODE}. This solution is monotone and bounded for all $t \in [0,T]$. 

\item\label{item:convergence_to_equi_ODE} As $t\to\infty$, the solution $u^{(\theta)}_t$ converges exponentially fast to its equilibrium $\Eq^{(\theta)} := (\frac{\theta}{c})^{1/p}$:
\begin{equation*}
|u^{(\theta)}_t - \Eq^{(\theta)}|\le |u_0 -\Eq^{(\theta)}|\, e^{- c p (\Eq^{(\theta)}\wedge u_0)^{p-1}t}.
\end{equation*}

\item\label{lem:distance-between-ode-solutions} For any other parameter  $\alpha>0$ we have 
\begin{equation}\label{eq:distance-sols-ode}
|\uu^{(\alpha)}_t- \uu^{(\theta)}_t|\leq | \alpha-\theta|\frac{1-e^{-c (u_0\wedge \Eq^{(\alpha)})^{p-1}\, t}}{c (u_0\wedge \Eq^{(\alpha)})^{p-1}}.
\end{equation}
Furthermore,  if $\alpha < \theta$, then $\uu^{(\alpha)}_t< \uu^{(\theta)}_t$, for all $t\in[0,T]$.
\end{enumerate}

\end{lemma}

\begin{proof}

Let $f(u) = \theta - c u^p$. Clearly, $f\in\mathcal{C}^1([0,\infty);\R)$, hence it is locally Lipschitz continuous in $\R_+$. By the Picard-Lindel\"of Theorem (see, e.g., \cite[Sec. 2.2]{perko2013differential}), there exists a unique local solution $(u^{(\theta)}_t)_{0\le t\le t_f}$ for the autonomous system \eqref{eq:ODE}, for some time $t_f>0$. 

Let $t^*=\inf\{t>0:\ u_t = 0\}$ and assume $t^* <\infty$. Then, from continuity of $f$
\[\lim_{t\uparrow t^*}f(u_t) = \theta>0,\]
i.e. there exists $\delta>0$ such that
\[\dfrac{d\uu_t}{dt}>\theta/2,\quad \forall t^*-\delta<t<t^*.\]

By integrating, we obtain that 
\[ -\frac{\theta}2(t^*-t)> u_t,\quad \forall t^*-\delta<t<t^*,\] 
which contradicts the fact that the trajectory approaches the value 0 from above. Hence, $u^{(\theta)}_\cdot$ is strictly positive in $[0,t_f]$.

We now show that the local solution is in fact global whenever $u_0>0$, which concludes the proof of \ref{item:wellposedness_ODE}. To this aim, we consider the sets
    \begin{align*}
    A_{+}\;(\text{Growth}) &=
        \bigl\{(t,\uu_t)\in\mathbb{R}_{+}\times\mathbb{R}_{+}:\;
               \theta > c\,\uu^{p}_t\bigr\},\\[2pt]
    A_{-}\;(\text{Decay}) &=
        \bigl\{(t,\uu_t)\in\mathbb{R}_{+}\times\mathbb{R}_{+}:\;
               \theta < c\,\uu^{p}_t\bigr\}.
\end{align*}

Then, if $(0,u_0)\in A_+$, $\dfrac{d\uu_t}{dt}>0$ and $\dfrac{d^2\uu_t}{dt^2}=-\,p\,c\,\uu^{p-1}_t\,\dfrac{d\uu_t}{dt} < 0$, implying that the solution is increasing and concave. Similarly, if $(0,u_0)\in A_-$, the solution is decreasing and convex. Moreover, when $u_0 = \Eq=\big(\frac{\theta}{c}\big)^{1/p}$, then the solution is constant $u_t\equiv \Eq$. In all cases, the solution $u^{(\theta)}$ is monotone and remains bounded by its initial value $u_0$ and its equilibrium point \(\Eq\). More precisely, $\uu_t^{(\theta)}$ stays in the compact $[\min(u_0, \Eq), \max(u_0, \Eq)]\subset\R_+$, for all $t\in[0,t_f]$, thus no finite-time blow-up can occur. By the standard continuation theorem for ODEs, the unique local solution extends to a unique global solution defined for all $t \ge 0$, and this solution satisfies $u_t\rightarrow\Eq$ as $t\to\infty$.

In order to prove \ref{item:convergence_to_equi_ODE}, let us consider $\tilde{u}_t^{(\theta)}:= \uu^{(\theta)}_t - \Eq^{(\theta)}$. Then, $\tilde{\uu}_t^{(\theta)}\to0$ as $t\to\infty$, 
and from Taylor expansion:
\begin{align*}
({\uu}_t^{(\theta)})^p &= (\Eq^{(\theta)})^p + p\,  \xi^{p-1}_t \, \tilde{\uu}_t^{(\theta)},
\end{align*}
for some $\xi_t$ between $\uu_t^{(\theta)}$ and $\Eq^{(\theta)}$.

Since $\theta - c\, (\Eq^{(\theta)})^p = 0$, we get the linearized equation around the equilibrium:
\begin{equation*}
\dfrac{d}{dt}\tilde{u}_t^{(\theta)} = -c\,p\,  \xi_t^{p-1} \, \tilde{\uu}_t^{(\theta)},\qquad \tilde{\uu}_0 = \uu_0 -\Eq^{(\theta)},
\end{equation*}
having solution
\begin{equation*}
\tilde{u}_t^{(\theta)} = (u_0-\Eq^{(\theta)}) \, e^{-c\, p\, \int_0^t \xi_s^{p-1}ds}, 
\end{equation*}

When $u_0>\Eq^{(\theta)}$, we have that $\uu_t\ge \xi_s\ge \Eq^{(\theta)}$. Otherwise, when $u_0<\Eq^{(\theta)}$, we have $u_0\le u_t < \Eq^{(\theta)}$ for all $t\ge0$. Thus,
\begin{align*}
|\tilde{\uu}_t^{(\theta)}|&\le |u_0-\Eq^{(\theta)}| \ e^{-cp \,\left(\min\{ u_0,\Eq^{(\theta)}\}\right)^{p-1}\, t},
\end{align*}
as desired.

To prove assertion \ref{lem:distance-between-ode-solutions} we follow a similar strategy by linearizing the ODE. 

Denote $\Delta u_t := u_t^{(\theta)} -u_t^{(\alpha)}$, and notice that
\begin{align*}
(u_t^{(\theta)})^p - (u_t^{(\alpha)})^p&= \Delta u_t\, h_t,
\end{align*}
where $h_t = p\int_0^1 (z u_t^{(\theta)} + (1-z) u_t^{(\alpha)})^{p-1}dz$. 

Then, since the solutions $u^{(\theta)}$ and $u^{(\alpha)}$ are positive, it is easy to check that 
\begin{equation}\label{eq:sol_compar}(u_0\wedge \Eq^{(\alpha)})^{p-1}\leq (\uu_t^{(\alpha)})^{p-1} \leq h_t\leq p(\uu_t^{(\theta)}\vee \uu_t^{(\alpha)})^{p-1}.
\end{equation}

Considering the linear equation
\begin{equation*}
\dfrac{d\Delta u_t}{dt} = (\theta-\alpha) - c \Delta u_t h_t,\qquad \Delta u_0 = 0,
\end{equation*}
having solution $\Delta u_t = (\theta-\alpha)\int_0^t \exp\{-c\int_s^t h_rdr\} ds,$
from \eqref{eq:sol_compar}  it is clear that
\begin{align*}
|\Delta u_t| &\leq |\theta-\alpha|\int_0^t \exp\{-c\, (u_0\wedge \Eq^{(\alpha)})^{p-1} \,(t-s)\} ds.
\end{align*}
Furthermore,  when $\theta> \alpha$, we have $u_t^{(\theta)} -u_t^{(\alpha)}> 0$, which concludes the proof.
\end{proof}

\subsection{On the non-linear process}
\paragraph{Proof of Proposition \ref{prop:well-posedness-and-properties-NLSDE} }

\begin{proof}[Proof of Proposition \ref{prop:well-posedness-and-properties-NLSDE}] Let $\uu$ the solution of the Equation \eqref{eq:ODE}, with initial condition $u_0>0$, and consider the SDE:
\begin{equation}
    X^{(\uu)}_t = X^{(\uu)}_0 + \int_0^t \big(\theta-c(\uu_s)^{p-1}X^{(\uu)}_s\big)ds + \int_0^t\sigma X^{(\uu)}_sdW_s, \quad X^{(\uu)}_0 = X_0.
\end{equation} 

The unique strong solution of this SDE is given by (see e.g. \cite[Sect. 5.6] {karatzas1991}) :
\begin{equation}\label{eq:explicitsolution}
\begin{aligned}
    X_t^{(\uu)} &=  X_0\exp\left(-\int_0^t\left(c  \uu_s^{p-1}+\frac{\sigma^2}{2}\right)ds + \sigma W_t \right)\\
    &\quad  + \theta \int_0^t \exp\left(-\int_s^t \left(c \uu_\tau^{p-1}+\frac{\sigma^2}{2}\right)d\tau + \sigma (W_t-W_s) \right)  ds,
\end{aligned}
\end{equation}
which is $\P$-a.s. strictly positive.
Moreover, taking expectations and applying Fubini-Tonelli:
\begin{equation}
    \E[X^{(\uu)}_t] = \E[X_0] + \int_0^t \big(\theta-c(\uu_s)^{p-1}\E[X^{(\uu)}_s]\big)ds.
\end{equation}

Notice that if we impose the initial ODE value $u_0=\E[X_0]$ and consider $g(t) = \E[X^{(\uu)}_t] - u_t$, the we get $g(0)=0$ and from the chain rule:
\begin{equation*}
g^2(t) = -2\, c\, \int_0^t g^2(s)\, (u_s)^{p-1}ds.
\end{equation*}

From Lemma \ref{prop:ode_properties}-\ref{item:wellposedness_ODE} we conclude $0\le g^2(t)\le 0$, implying $\E[X^{(\uu)}] = u_t$, for all $t\ge0$, i.e. $ X^{(\uu)}$ solves Equation \eqref{eq:ODE}.

To analyse the finiteness of the moments it is enough to identify $X^{(\uu)}$ as in \eqref{eq:explicitsolution}. Indeed, for $\gamma\ge1$ we have from Jensen inequality
\begin{align*}
(X_t^{(\uu)})^{\gamma} &= \exp \left\{-c\gamma\int_0^t  \uu_s^{p-1}ds - \frac{\sigma^2}{2}\gamma t + \sigma\gamma  W_t \right\} \\
& \quad \times \left( X_0 + \theta \int_0^t \exp \left\{c\int_0^s  \uu_r^{p-1}dr + \frac{\sigma^2}{2} s - \sigma W_s \right\} ds \right)^{\gamma}\\
&\le 2^{\gamma-1}\,\exp \left\{\sigma\gamma  W_t \right\} \left( X_0^{\gamma} + \theta^\gamma t^{\gamma-1} \int_0^t \exp \left\{c\gamma\int_0^s  \uu_r^{p-1}dr + \frac{\sigma^2}{2}\gamma s - \sigma\gamma W_s \right\} ds \right).
\end{align*}

Hence, from Lemma \ref{prop:ode_properties}, there exists a deterministic constant $C$, depending in particular on $\Eq^{(\theta)}$, such that
\begin{align*}
\E\Big[\sup_{0\le t\le T}(X_t^{(\uu)})^{\gamma}
\Big]&\le C\, \E\Big[\sup_{0\le t\le T}\left( X_0^{\gamma}\exp \left\{\sigma\gamma  W_t \right\} + \int_0^t\exp \left\{\sigma\gamma  (W_t-W_s) \right\} ds \right)\Big].
\end{align*}

From independence between $X_0$ and $W$ and Doob's martingale inequality, if $X_0\in L^\gamma(\Omega)$, then the first term is finite; similarly for the second term, obtaining:  
\begin{align*}
\E\Big[\sup_{0\le t\le T}(X_t^{(\uu)})^{\gamma}
\Big]&<\infty.
\end{align*}
This property can be extended to $\gamma\in(0,1)$ by applying H\"older inequality.

Similarly, for negative moments we have that there exists a deterministic constant $C>0$, such that: 
\begin{align*}
(X_t^{(\uu)})^{-\gamma} &= \exp \left\{c\gamma\int_0^t  \uu_s^{p-1}ds + \frac{\sigma^2}{2}\gamma t - \sigma\gamma  W_t \right\} \\
& \quad \times \left( X_0 + \theta \int_0^t \exp \left\{c\int_0^s  \uu_r^{p-1}dr + \frac{\sigma^2}{2} s - \sigma W_s \right\} ds \right)^{-\gamma}\\
&\le C\, \exp \left\{- \sigma\gamma  W_t \right\} X_0^{-\gamma}.
\end{align*}

Then, from independence between $X_0$ and $W$ and Doob's martingale inequality we get \[\E\Big[\sup_{0\le t\le T}(X_t^{(\uu)})^{-\gamma}
\Big]<\infty.\]

Finally, since $\E[X_t^{\uu}] = \uu_t \in [u_0\wedge\Eq,\, \uu_0\vee\Eq]$ for all $t\ge0$,  property \eqref{eq:uniform-first-moment} is straightforward. 
\end{proof}

\paragraph{A digression on linear stochastic differential equations  }

Let $a=c\alpha_0^{p-1}$, for some positive constant $\alpha_0$. Then, it follows that the process defined by  
\begin{equation}\label{eq:linear-sde-cc}
dX_t = (\theta - aX_t)dt + \sigma X_t dW_t,\;\; X_0 = x_0>0,
\end{equation}
is  Harris positive recurrent with invariant measure $\pi^{\theta,\sigma}_a$ given by an Inverse Gamma distribution with shape parameter   $1 + \frac{2a}{\sigma^2}$  and scale parameter $\frac{2\theta}{\sigma^2}$ (see, e.g., \cite{10.3150/bj/1116340291,zbMATH00195091}).

In particular, according to \cite[Thm. 1.6]{EvaLocherbach2015}, for any measurable, positive functions $f$ and $g$ with $0 <\int g(u) \pi^{\theta,\sigma}_a(du) <+ \infty$, and any $X_0 = x\in\R_+$:
\begin{equation}\label{eq:ergodig-means-linear-sde}
\lim_{t\to\infty} \frac{\int_0^t f(X_s)ds}{\int_0^t g(X_s)ds} = \frac{\int f(u) \pi^{\theta,\sigma}_a(du)}{\int g(u)\pi^{\theta,\sigma}_a(du)} \quad \P_x\text{-a.s.}
\end{equation}

\begin{remark}[Negative Moments of the Inverse Gamma Law]
\label{rem:neg_moments_bounding}
If $\pi_{\alpha}^\beta$  is a Inverse Gamma law with shape parameter $\alpha$ and scale parameter $\beta$ , then for  any integer $r \ge 1$ it holds
\begin{equation}\label{eq:negative-moments-inverse-gamma}
 \int \frac{1}{u^r} \pi^{\beta}_\alpha(du)= \frac{\alpha(\alpha+1)\cdots(\alpha+r-1)}{\beta^r}.
\end{equation}

\end{remark}

The following lemma exploits the ergodicity property of linear SDEs together with the characterization of $\E[X_t]$ from Lemma \ref{prop:ode_properties} to establish the asymptotic behavior of the ergodic moments of the solution to \eqref{eq:SM1}.

\begin{lemma}\label{lem:Q12}
Let $(X_t)_{0\le t\ge T}$ be the solution of Equation \eqref{eq:SM1}. 
\begin{enumerate}[label= (\roman*)]
\item \label{lem:Q12-i} For $u_{eq}=\Big(\frac{{\theta}}{c}\Big)^{1/p}$, we have that
\begin{equation}\label{eq:value_Q12}
Q_{12}:=\lim_{t\to+\infty} \frac{\int_0^t \frac{ds}{X_s}}{\int_0^t \frac{ds}{X_s^2}} = \frac{\theta}{\sigma^2+c(\Eq)^{p-1}}.
\end{equation}
\item \label{lem:Q12-ii} For any bounded measurable $g:[0,\infty)\to[0,\infty)$ such that
$
\bar{g}:=\lim_{t\to+\infty}g(t)<\infty,
$
we have
\begin{equation}\label{eq:gQ12}
\begin{aligned}
\lim_{t\to+\infty} \frac{\int_0^t \frac{g(s)ds}{X_s}}{\int_0^t \frac{ds}{X_s^2}} = \bar{g}Q_{12}.
\end{aligned}
\end{equation}

\end{enumerate}
\end{lemma}

\begin{proof}

Let $\epsilon>0$,  and define 
$$
\bar{t}_\epsilon = \frac{1+\epsilon}{cp\left(\Eq\wedge u_0\right)^{p-1}}\log\left(\frac{\left|\uu_0-\Eq\right|}{\epsilon}\right),
$$
where $\Eq$ is the limit point of the solution to the ODE \eqref{eq:ODE}.

From Lemma \ref{prop:ode_properties}-\ref{item:convergence_to_equi_ODE}, for all $t\geq \bar{t}_\epsilon$, we get $|\uu_t-\Eq|<\epsilon$. Moreover, for all $t\geq \bar{t}_\epsilon$ and $x\geq0$,
$$
b^\sharp(x):=\theta-c(\Eq - \epsilon)^{p-1}x \geq \theta-cu^{p-1}_tx\geq \theta -c(\Eq+ \epsilon)^{p-1}x =: b^\flat(x).
$$

Let $X^\sharp$ (resp. $X^\flat$) the solution of the SDE with $b^\sharp$ (resp. $b^\flat$), same diffusion as in Equation \eqref{eq:SM1}, and initial condition
$X^\sharp_{\bar{t}_\epsilon}=X_{\bar{t}_\epsilon}$ (resp. $X^\flat_{\bar{t}_\epsilon}=X_{\bar{t}_\epsilon}$). Then, recalling that $(u_t)_{t\ge0}$ is bounded and applying a comparison theorem for SDEs (see, e.g., \cite[Prop. 2.18]{karatzas1991}), we have:
\begin{align*}
 {X^\sharp_s}^{-1} &\leq     {X_s}^{-1}   \leq {X^\flat_s}^{-1},\quad \text{ for all }\quad s\geq \bar{t}_\epsilon.
\end{align*}

Hence, for any $t\geq \bar{t}_\epsilon$, and any $\gamma>0$:
\begin{align}\label{eq:bounds_sost_bemol}
\int_0^{\bar{t}_\epsilon} \frac{ds}{X_s^\gamma} + \int_{\bar{t}_\epsilon}^t \frac{ds}{(X^\sharp_s)^\gamma} &\leq   \int_0^{t} \frac{ds}{X_s^\gamma}   \leq \int_0^{\bar{t}_\epsilon} \frac{ds}{X_s^\gamma} +\int_{\bar{t}_\epsilon}^t \frac{ds}{(X^\flat_s)^\gamma}.
\end{align}

Notice that $\int_0^{\bar{t}_\epsilon} X_s^{-\gamma}ds<\infty$ almost surely, since we are integrating an almost surely continuous and positive function in a deterministic compact interval.

Then, taking $\gamma=1,2$ in \eqref{eq:bounds_sost_bemol}, we have:
\begin{align*}
\limsup_{t\to+\infty}\frac{1}{t}\int_0^t \frac{ds}{X_s}   
  & \le\lim_{t\to+\infty} \frac{1}{t}\int_0^{\bar{t}_{\epsilon}} \frac{ds}{X_s} +\limsup_{t\to+\infty} \frac{1}{t}\int_{\bar{t}_{\epsilon}}^t \frac{ds}{X^\flat_s}     \\
  & =  \lim_{t\to+\infty} \frac{t-\bar{t}_{\epsilon}}{t}\limsup_{t\to+\infty} \frac{1}{t-\bar{t}_{\epsilon}}\int_{\bar{t}_{\epsilon}}^t \frac{ds}{X^\flat_s},
\end{align*}
and similarly
\begin{align*}
\liminf_{t\to+\infty}\frac{1}{t}\int_0^t \frac{ds}{X_s^2}   &\geq\liminf_{t\to+\infty} \frac{1}{t-\bar{t}_{\epsilon}}\int_{\bar{t}_{\epsilon}}^t \frac{ds}{\left(X^\sharp_s\right)^2}.
\end{align*}

According to \eqref{eq:ergodig-means-linear-sde}, with $g(x)=1$, and Remark \ref{rem:neg_moments_bounding}:

\begin{align*}
\lim_{t\to\infty} \frac{1}{t}\int_0^t (X_s^\flat)^{-\gamma}ds &= \frac{(\frac{2c\,(\Eq+\epsilon)^{p-1}}{\sigma^2}+1)(\frac{2c\,(\Eq+\epsilon)^{p-1}}{\sigma^2}+2)\cdots(\frac{2c\,(\Eq+\epsilon)^{p-1}}{\sigma^2}+\gamma)}{\left(\frac{2\theta}{\sigma^2}\right)^\gamma},\\
\lim_{t\to\infty} \frac{1}{t}\int_0^t (X_s^\sharp)^{-\gamma}ds &= \frac{(\frac{2c\,(\Eq-\epsilon)^{p-1}}{\sigma^2}+1)(\frac{2c\,(\Eq-\epsilon)^{p-1}}{\sigma^2}+2)\cdots(\frac{2c\,(\Eq-\epsilon)^{p-1}}{\sigma^2}+\gamma)}{\left(\frac{2\theta}{\sigma^2}\right)^\gamma}.
\end{align*}

Putting everything together, we obtain
\begin{align*}
\limsup_{t\to+\infty} \frac{\int_0^t \frac{ds}{X_s}}{\int_0^t \frac{ds}{X_s^2}} & \leq {2\theta}\, \frac{2c(\Eq+ \epsilon)^{p-1}+\sigma^2}{\left(2c(\Eq- \epsilon)^{p-1}+\sigma^2\right)\left(2c(\Eq- \epsilon)^{p-1}+2\sigma^2\right)}.
\end{align*}
The left-hand side does not depend on $\epsilon$, while the right-hand side is continuous in $\epsilon$. Hence, letting $\epsilon\rightarrow0$, we get
\begin{align*}
\limsup_{t\to+\infty} \frac{\int_0^t \frac{ds}{X_s}}{\int_0^t \frac{ds}{X_s^2}}  &\leq   \frac{\theta}{c(\Eq)^{p-1}+\sigma^2},\, \P\text-{a.s.}
\end{align*}

An analogous argument yields
\begin{align*}
\liminf_{t\to+\infty} \frac{\int_0^t \frac{ds}{X_s}}{\int_0^t \frac{ds}{X_s^2}} &\geq 2\theta\, \frac{ 2c(\Eq- \epsilon)^{p-1}+\sigma^2}{\left(2c(\Eq+ \epsilon)^{p-1}+\sigma^2\right)\left(2c(\Eq+ \epsilon)^{p-1}+2\sigma^2\right)}.
\end{align*}

Thus, letting $\epsilon\to 0$ gives
$$
\lim_{t\to+\infty} \frac{\int_0^t \frac{ds}{X_s}}{\int_0^t \frac{ds}{X_s^2}} = \frac{\theta}{c(\Eq)^{p-1} + \sigma^2},
$$
as desired.\\

In order to prove \ref{lem:Q12-ii}, we write
\begin{align*}
  \frac{\int_{0}^t \frac{g(s)}{X_s}ds}{\int_0^t  \frac{1}{X_s^2}ds}  
  &=     \frac{\int_{0}^t \frac{\bar{g}}{X_s}ds}{\int_0^t  \frac{1}{X_s^2}ds}+\frac{\int_{0}^t\frac{g(s)-\bar{g}}{X_s}ds}{\int_0^t  \frac{1}{X_s^2}ds}=: \bar{g}Q_{12}(t) + R(t).
\end{align*}

Then, from \ref{lem:Q12-i}, it remains to show that $\lim_{t\to\infty}R(t)=0$, $\P$-almost surely.  To this end, let $t_\epsilon$ such that for all $t\geq t_\epsilon$, $|g(t)-\bar{g}|<\epsilon$. It follows that
\begin{align*}
 |R(t)| &\leq  \|g\|_\infty\frac{\int_0^{t_\epsilon} \frac{1}{X_s}ds}{\int_0^t  \frac{1}{X_s^2}ds} + \epsilon\frac{\int_{ 0}^t \frac{1}{X_s}ds}{\int_0^t  \frac{1}{X_s^2}ds} - \epsilon\frac{\int_0^{ t_\epsilon} \frac{1}{X_s}ds}{\int_0^t  \frac{1}{X_s^2}ds}=:R_{1}(t)+R_{2}(t)-R_{3}(t).
\end{align*}

Since $\int_0^{ t_\epsilon} X_s^{-1}ds<\infty$ almost surely, and the denominator in  $R(t)$ diverges as $t$ goes to infinity (see the proof of \ref{lem:Q12-i}), we obtain:
$
\lim_{t\to\infty}|R_{1}(t)|+|R_{3}(t)| = 0.
$
Meanwhile, by definition, $
\lim_{t\to\infty}R_2(t) = \epsilon\  Q_{12}$, which yields
$$
\limsup_{t\to\infty} |R(t)| \leq \epsilon\,  Q_{12}.$$
 Taking $\epsilon\to0$ and using that $\P$-a.s. $Q_{12}<\infty$, we conclude.
\end{proof}

\begin{lemma}\label{lem:convergence-of-QuoT} Let $\quo(T)$ defined by
$$
\quo(T):= \frac{\int_0^T\frac{dX_t}{X_t^2}}{\int_0^T\frac{1}{X_t^2}dt},
$$
as in the proof of Lemma \ref{cor:the-estimator-is-eventually-positive}. Then:
\begin{enumerate}[label= (\roman*)]
\item  \label{lem:convergenceAndpositivityQ} It holds,
$$\lim_{T\to\infty}\quo(T) = \frac{\theta\sigma^2}{c^{1/p}\theta^{1-1/p}+\sigma^2},\quad \P\text{-a.s.} $$
Moreover, there exists a random variable $\TT$ such that, for all $T\geq \TT$, $\quo(T)$ is bounded away from zero.
\item \label{lem:tclQ} A central limit theorem for $\quo(T)$ is satisfied:
 $$\sqrt{\frac{ T}{\varInfty}}\left( \quo(T)-\theta+\frac{c\int_{0}^T \frac{\uu_s^{p-1}}{X_s}ds}{\int_0^T  \frac{1}{X_s^2}ds}\right) \Rightarrow \mathcal{N}(0,1),\text{ as $T\to\infty$},$$
where
\begin{equation*}
\varInfty = \frac{2\theta^2\sigma^2}{(2c^{1/p}\theta^{1-1/p}+\sigma^2)(c^{1/p}\theta^{1-1/p}+\sigma^2)}.
\end{equation*}
\end{enumerate}

\end{lemma}

\begin{proof}
Define the process $(M_t:=\int_0^t  \frac{1}{X_s}dW_s)_{t\ge0}$. From Proposition \ref{prop:well-posedness-and-properties-NLSDE} we know that $M$ is a continuous martingale. Moreover, from the proof of Lemma \ref{lem:Q12} we have $\P$-almost surely that $\langle M \rangle_t = \int_0^T  \frac{1}{X_s^2}ds\to\infty$ when $t\to\infty$. Hence, thanks to the strong law of large numbers for martingales (see, e.g. \cite[Theorem 3.4]{mao2007stochastic})
$$
\frac{M_T }{\langle M\rangle_T } \to 0, \quad \P\text{-a.s. when $T\to\infty$.}
$$
From this and Lemma \ref{lem:Q12}-\ref{lem:Q12-ii}, we get

$$
\lim_{T\to\infty}\quo(T)= \theta - c\lim_{T\to\infty}\frac{\int_{0}^T \frac{\uu_s^{p-1}}{X_s}ds}{\int_0^T  \frac{1}{X_s^2}ds}=\theta -c\uu_{eq}^{p-1}Q_{12}.
$$
From Lemma \ref{lem:Q12}-\ref{lem:Q12-i} it follows that
 $$
\lim_{T\to\infty} \quo(T)= \frac{\theta\sigma^2}{c^{1/p}\theta^{1-1/p}+\sigma^2}>0.
 $$

The second part of the \ref{lem:convergence-of-QuoT}-\ref{lem:convergenceAndpositivityQ}  follows inmediately. 

To prove \ref{lem:tclQ} we rewrite
\begin{equation*}
\begin{aligned}
  \frac{\sqrt{\langle M \rangle_T}}{\sigma}\left( \quo(T)-\theta+\frac{c\int_{0}^T \frac{\uu_s^{p-1}}{X_s}ds}{\int_0^T  \frac{1}{X_s^2}ds}\right)& =     \frac{M_T}{\sqrt{\langle M \rangle_T}}.
  \end{aligned}
\end{equation*}
Then, thanks to the Central Limit Theorem for martingales \cite[Thm 6.31]{hausler2015stable} we have that
$$
  \frac{\sqrt{\langle M \rangle_T}}{\sigma}\left( \quo(T)-\theta+\frac{c\int_{0}^T \frac{\uu_s^{p-1}}{X_s}ds}{\int_0^T  \frac{1}{X_s^2}ds}\right) \Rightarrow \mathcal{N}(0,1),\text{ as $T\to\infty$}.
$$
We identify from Remark \ref{rem:neg_moments_bounding} that:
$$
\lim_{T \to\infty} \sqrt{ \frac{\langle M \rangle_T}{T} } 
= \lim_{T \to\infty}\sqrt{ \frac{1}{T}\int_0^T \frac{1}{X_t^2}dt} = \sqrt{\frac{\sigma^2}{\varInfty}},\quad \text{$\P$-almost surely}.
$$
Thanks to Slutsky's theorem, we have
$$
 \sqrt{\frac{ T}{\varInfty}}\left( \quo(T)-\theta+\frac{c\int_{0}^T \frac{\uu_s^{p-1}}{X_s}ds}{\int_0^T  \frac{1}{X_s^2}ds}\right) \Rightarrow \mathcal{N}(0,1),\text{ as $T\to\infty$},
$$
where
$$
\varInfty = \frac{2\theta^2\sigma^2}{(2c^{1/p}\theta^{1-1/p}+\sigma^2)(c^{1/p}\theta^{1-1/p}+\sigma^2)}.
$$
\end{proof}

\subsection{On the continuously sampled estimator}\label{ap:proofs-convergence-surrogate-estimators}

To address the convergence of the continuously sampled estimator, we introduce a {\it proxy}, which consist in considering to observe the process at large times.   
\begin{definition}
The asymptotic-in-time estimator $\hat{\theta}^{(k)}$, is given by $\hat{\theta}^{(0)}=0$ and
\begin{equation}\label{eq:second-surrogate-estimator}
\begin{aligned}
    \hat{\theta}^{(k)}& = \theta  + \lim_{T\to\infty}\frac{c\int_{0}^T \frac{({u}_{s}^{(k)})^{p-1}-\uu_s^{p-1}}{X_s}ds}{\int_0^T  \frac{1}{X_s^2}ds},\\
    u_t^{(k)} &= \uu_0+ \int_0^t\Big(\hat{\theta}^{(k-1)}-c\left(u^{(k)}_s\right)^p\Big)ds.
\end{aligned}
\end{equation}
\end{definition}

\begin{lemma}\label{lemma:integrals-dynamics-vs-equilbrium}
Let $(X_t)_{0\le t\ge T}$ be the solution of Equation \eqref{eq:SM1} and $(\hat{\theta}^{(k)},\, \uu^{(k)})$ be the asymptotic-in-time estimator with initialization $\hat{\theta}^{(0)}=0$.
Then, for all $k\geq1$,  $\hat{\theta}^{(k)}$ is in fact deterministic and strictly positive. Moreover, it satisfies the recurrence relation
$$
\hat{\theta}^{(k)} = \theta \frac{c^{1/p}(\hat{\theta}^{(k-1)})^{1-1/p}+\sigma^2}{c^{1/p}\theta^{1-1/p}+\sigma^2}.$$
\end{lemma}
\begin{proof} 
We will proceed inductively. We first prove the result for $k=1$. Indeed, since $\hat{\theta}^{(0)}=0$, $u^{(1)}$ is explicitly computable, and is given by
$$
u^{(1)}_t = \frac{1}{\left[u_0^{-(p-1)}+c(p-1)t\right]^{1/(p-1)}},
$$
and $u_{eq}^{(1)}=0$. Then, thanks to \eqref{eq:gQ12} with $g(t)=c(({u}_{t}^{(1)})^{p-1}-\uu_t^{p-1})$ we have

\begin{equation*}
\hat{\theta}^{(1)} = \theta -c\uu_{eq}^{p-1}Q_{12} = \theta\frac{\sigma^2}{c(\Eq)^{p-1}+\sigma^2},
\end{equation*}
which is deterministic and strictly positive, as desired.

Now we proceed by induction. Let us assume that $ \hat{\theta}^{(k-1)}$ is deterministic and strictly positive, therefore $({u}^{(k)})^{p-1}-\uu^{p-1}$ satisfies the hypothesis of Lemma \ref{lem:Q12}-\ref{lem:Q12-ii} and then
$$
\hat{\theta}^{(k)} = \theta + c((u_{eq}^{(k)})^{p-1}-\uu_{eq}^{p-1})Q_{12} = \frac{\theta\sigma^2}{c\Eq^{p-1}+\sigma^2}+\frac{c\theta(\Eq^{(k)})^{p-1}}{c\Eq^{p-1}+\sigma^2}.
$$

Since $u_{eq}^{(k)}$ depends only on $\theta^{(k-1)}$, which is positive and deterministic, we conclude that $\theta^{(k)}$ is also positive and deterministic.
\end{proof}

In the what follows, we will need to control the variations of the function $x^r$ for different values of $r>0$.

\begin{lemma} [On the locally-Lipschitz property of $x^r$]\label{rem:on-x-to-power-r} Let $r>0$, and $x,y\in(0,\infty)$. Then it holds
$$
|x^r - y^r|   \leq  L_r(x,y)|x-y|,
$$
where
\begin{equation}\label{eq:def-L-r}
L_r(x,y)=\left\{
\begin{array}{cc} 
1/y^{1-r}&r\in(0,1)\\
r\max\{x,y\}^{r-1} & r>1.
\end{array}
\right.
\end{equation}

\end{lemma}

\begin{proof} Notice that
$$
x^r - y^r = r(x-y)\int_0^1 (\lambda x + (1-\lambda)y)^{r-1}d\lambda.
$$
Therefore, for $r\in(0,1)$,
\begin{align*}
|x^r - y^r|   &\leq |x-y|\int_0^1  \frac{rd\lambda}{(\lambda x + (1-\lambda)y)^{1-r}}      \le   \frac{|x-y|}{y^{1-r}}.
\end{align*}
Meanwhile, for $r>1$, we have 
$$
|x^r - y^r|   \leq r\max\{x,y\}^{r-1} |x-y|.
$$
\end{proof}

\begin{lemma}
As $k$ goes to infinity the asymptotic-in-time estimator converges monotonically towards $\theta$.
\end{lemma}

\begin{proof}
From Lemma \ref{lemma:integrals-dynamics-vs-equilbrium} we know that
\begin{align}\label{eq:diff_thetak_theta}
  \hat{\theta}^{(k)} - \theta  
  & =  c^{1/p}\left(\left( \hat{\theta}^{(k-1)}\right)^{1-1/p}-\theta^{1-1/p}\right)Q_{12}.
\end{align}
On the other hand, from Lemma \ref{rem:on-x-to-power-r}, since $\theta$ and $\hat{\theta}^{(k-1)}$ are positives and $p\ge1$:
\begin{align*}
\left| \left( \hat{\theta}^{(k-1)}\right)^{1-1/p} - \theta^{1-1/p} \right|
  &\le \frac{|\hat{\theta}^{(k-1)}-\theta|}{\theta^{1/p}}.
\end{align*}
Therefore, we obtain,
$$
|  \hat{\theta}^{(k)}(\omega) - \theta | \leq \frac{Q_{12}}{\uu_{eq}}|\hat{\theta}^{(k-1)}(\omega)-\theta|,
$$
and iterating this inequality we get:
\begin{equation}\label{eq:geometric-convergence-theta-k}
|  \hat{\theta}^{(k)}(\omega) - \theta | \leq \left(\frac{Q_{12}}{\uu_{eq}}\right)^k|\theta|,
\end{equation}
from where we can conclude as $Q_{12}/\uu_{eq}<1$. Indeed, from \eqref{eq:value_Q12},
$$
\frac{Q_{12}}{\Eq}  = \frac{\theta}{(\sigma^2+c(\Eq)^{p-1})\Eq}  < \frac{\theta}{c\left(\Eq\right)^{p}}=1.
$$

Finally, from \eqref{eq:diff_thetak_theta} we have
\[
\hat{\theta}^{(k)}-\hat{\theta}^{(k-1)}
= c^{1/p}\Big(\big(\hat{\theta}^{(k-1)}\big)^{1-1/p}
      -\big(\hat{\theta}^{(k-2)}\big)^{1-1/p}\Big)Q_{12}.
\]

Hence, $\hat{\theta}^{(k)}-\hat{\theta}^{(k-1)}$ has the same sign than $\hat{\theta}^{(k-1)}-\hat{\theta}^{(k-2)}$. Moreover, from Lemma \ref{lemma:integrals-dynamics-vs-equilbrium}, $\hat{\theta}^{(1)}>0=\hat{\theta}^{(0)}$, which implies in particular that $\hat{\theta}^{(2)}-\hat{\theta}^{(1)}>0$.  By induction, 
$\hat{\theta}^{(k)}-\hat{\theta}^{(k-1)}>0$ for all $k\ge1$. 
Hence $\{\hat{\theta}^{(k)}\}_{k\ge0}$  converges to $\theta$ monotonically.

\end{proof}
\begin{remark}\label{rem:geometric-convergence-theta-k}
Notice that the convergence of $\hat{\theta}^{(k)}$ towards $\theta$ is geometric with rate given in \eqref{eq:geometric-convergence-theta-k} by $\theta^{1-1/p}c^{1/p}/(\sigma^2+\theta^{1-1/p}c^{1/p})$. This hints that smaller values of $\theta$ or larger values of $\sigma^2$ should lead to faster convergence of the whole EM algorithm on $k$.
\end{remark}

Now we can proceed with the analysis of the continuously sampled estimator. We start with the proof of Proposition \ref{prop:thetaT-is-eventually-positive}.

\begin{proof}[Proof of Proposition \ref{prop:thetaT-is-eventually-positive}]
We proceed by induction. For $k=1$,  since  $\hat{\theta}^{(0,T)}=0$ and $u^{(1,T)}$ is strictly positive:
\begin{equation}\label{eq:def-hat-theta-1-t}
\hat{\theta}^{(1,T)} = \quo(T)+\frac{c\int_{0}^T \frac{({u}_{s}^{(1,T)})^{p-1}}{X_s}ds}{\int_0^T  \frac{1}{X_s^2}ds}> 
\quo(T).
\end{equation}
Thus, from Lemmas \ref{lem:convergence-of-QuoT}-\ref{lemma:integrals-dynamics-vs-equilbrium}, we have
\begin{equation*}
\liminf_{T\to\infty}\hat{\theta}^{(1,T)}(\omega)> 
\hat{\theta}^{(1)}>0.
\end{equation*}
It follows directly that there exists a random variable $\TT(\omega)$ such that 
 $$
 \hat{\theta}^{(1,T)}(\omega) > \frac{\hat{\theta}^{(1)}}{2}>0,\;\forall\, T\geq \TT(\omega).
 $$
 
Now we proceed with the inductive step. Assume for $k\geq 2$, that $\hat{\theta}^{(k-1,T)}$ is positive and bounded away from zero for $T\geq \TT(\omega)$. Then $u^{(k,T)}$ is well defined and strictly positive, so we can compute $\hat{\theta}^{(k,T)}$  as
\begin{equation*}
\begin{aligned}
   \hat{\theta}^{(k,T)}& =\quo(T)+\frac{c\int_{0}^T \frac{({u}_{s}^{(k,T)})^{p-1}}{X_s}ds}{\int_0^T  \frac{1}{X_s^2}ds},
\end{aligned}
\end{equation*}
but we know that for $T\geq \TT$, $\quo(T)$ is bounded away from zero, and the second term in the right-hand side is no negative, so $ \hat{\theta}^{(k,T)}$ is also  bounded away from zero. 
\end{proof}

\begin{lemma}\label{lemma:a-priori-control-estimator-surrogate} 
The continuously sampled estimator  is almost surely bounded when $T\to\infty$.
\end{lemma}

\begin{proof} 
Let us assume that $\TT$ is as in Lemma \ref{prop:thetaT-is-eventually-positive} so that, for any $T\geq \TT$, we have $\hat{\theta}^{(k-1,T)}$ is strictly positive and ${u}_{\cdot}^{(k,T)}$ is well defined. Then, from Lemma \ref{prop:ode_properties} and Young inequality, 
\begin{align*}
  \hat{\theta}^{(k,T)} & \leq |\quo(T)| + \frac{c\int_{0}^T \frac{({u}_{s}^{(k,T)})^{p-1}}{X_s}ds}{\int_0^T  \frac{1}{X_s^2}ds}    \\
  & \leq |\quo(T)| +  c\left(u_0\lor \left(\frac{\hat{\theta}^{(k-1,T)} }{c}\right)^{1/p} \right)^{p-1}\frac{\int_{0}^T \frac{1}{X_s}ds}{\int_0^T  \frac{1}{X_s^2}ds}      \\
    & \leq  |\quo(T)| +  \left(u_0^{p-1}+ c^{-(p-1)} + \hat{\theta}^{(k-1,T)}\right)\frac{\int_{0}^T \frac{1}{X_s}ds}{\int_0^T  \frac{1}{X_s^2}ds}.
\end{align*}
Iterating this inequality, it follows:
$$
  \hat{\theta}^{(k,T)} \leq \left( |\quo(T)| +  \left(u_0^{p-1}+ c^{-(p-1)} \right)\frac{\int_{0}^T \frac{1}{X_s}ds}{\int_0^T  \frac{1}{X_s^2}ds}\right)\frac{\left(\frac{\int_{0}^T \frac{1}{X_s}ds}{\int_0^T  \frac{1}{X_s^2}ds}\right)^k-1}{\frac{\int_{0}^T \frac{1}{X_s}ds}{\int_0^T  \frac{1}{X_s^2}ds}-1}.
$$
Since k is fixed, and the right-hand side in the last bound is almost surely convergent as $T\to\infty$, we conclude that $  \hat{\theta}^{(k,T)}$ is almost surely bounded as $T\to\infty$.
\end{proof}

We can now study the convergence of the continuously sampled estimator when $T$ and $k$ go to infinity.
\begin{proof}[Proof of Proposition \ref{e2:convergence-in-T}]

Notice that from Definition \ref{def:surrogate_definition}
\begin{equation}\label{eq:bound-for-error-in-T}
\left| \hat{\theta}^{(k,T)} -     \hat{\theta}^{(k)}\right| \leq  
  R_1(T) +    \left|\frac{c\int_{0}^T \frac{({u}_{s}^{(k,T)})^{p-1}-(u_s^{(k)})^{p-1}}{X_s}ds}{\int_0^T  \frac{1}{X_s^2}ds}\right| + R_2(T,k),
\end{equation}
where 
\[R_1(T):= \lim_{T\to0}  \frac{\sigma\int_0^T  \frac{1}{X_s}dW_s }{\int_0^T  \frac{1}{X_s^2}ds }\]
and 
\[R_2(T,k):= \left|\frac{c\int_{0}^T \frac{({u}_{s}^{(k)})^{p-1}-\uu_s^{p-1}}{X_s}ds}{\int_0^T  \frac{1}{X_s^2}ds}- \lim_{T\to\infty}\frac{c\int_{0}^T \frac{({u}_{s}^{(k)})^{p-1}-\uu_s^{p-1}}{X_s}ds}{\int_0^T  \frac{1}{X_s^2}ds}\right|.\]

From the proof of Lemma \ref{lem:Q12} and the law of large numbers for martingales (see, e.g. \cite[Theorem 3.4]{mao2007stochastic}), we have that
$$\frac{\sigma\int_0^T  \frac{1}{X_s}dW_s }{\int_0^T  \frac{1}{X_s^2}ds } \to 0, \quad \text{almost surely when $T\to\infty$.}$$
Thus, 
$\lim_{T\to\infty} R_1(T)=0$ almost surely. Moreover, by definition, for all fixed $k\ge 1$, we also have $\lim_{T\to\infty} R_2(T,k)=0$.

For the second term in the right-hand side of \eqref{eq:bound-for-error-in-T}, we consider $L_p({u}_{t}^{(k,T)},{u}_{t}^{(k)})$ defined in \eqref{eq:def-L-r} such that
$$
|({u}_{t}^{(k,T)})^{p-1}-(u_t^{(k)})^{p-1}| \leq L_p({u}_{t}^{(k,T)},{u}_{t}^{(k)})| {u}_{t}^{(k,T)}-u_t^{(k)}|
$$
Notice that for all $p>1$, $L_p({u}_{t}^{(k,T)},{u}_{t}^{(k)})$ can be bounded independently of the time $t$ (denoted by $L_p(k,T)$ in that case), and we have then
\begin{align*}
\left| \hat{\theta}^{(k,T)} -     \hat{\theta}^{(k)}\right| &\leq  R_1(T) +    \frac{cL_p(k,T)\int_{0}^T \frac{|{u}_{s}^{(k,T)}-u_s^{(k)}|}{X_s}ds}{\int_0^T  \frac{1}{X_s^2}ds} + R_2(k,T).
\end{align*}
Applying Lemma \ref{prop:ode_properties}-\ref{lem:distance-between-ode-solutions}, there exists a finite constant $C(\hat\theta^{(k-1)}, \hat\theta^{(k-1,T)},u_0,\omega)>0$, such that
\begin{align*}
\left| \hat{\theta}^{(k,T)} -     \hat{\theta}^{(k)}\right| &\leq  R_1(T) + \\
& +  \frac{cL_p(k,T)\int_{0}^T \frac{1}{X_s}ds}{\int_0^T  \frac{1}{X_s^2}ds}C(\hat\theta^{(k-1)}, \hat\theta^{(k-1,T)},u_0,\omega)\left| \hat{\theta}^{(k-1,T)} -     \hat{\theta}^{(k-1)}\right| +R_2(k,T)\\
&=:R_1(T) +   M(k,T)\left| \hat{\theta}^{(k-1,T)} -     \hat{\theta}^{(k-1)}\right| + R_2(k,T).
\end{align*}
Iterating, and recalling $\hat{\theta}^{(0,T)} = \hat{\theta}^{(0)}=0$, we obtain
$$
 \left| \hat{\theta}^{(k,T)}- \hat{\theta}^{(k)}\right|  \leq R_1(T)\sum_{j=0}^{k-1}\prod_{l=0}^{j-1}M(k-l,T) + \sum_{j=0}^{k-1}R_2(k-j,T)\prod_{l=0}^{j-1}M(k-l,T),
$$
where an empty product is interpreted as $1$. Recall that $k$ is fixed. Hence, the right-hand side of the last inequality contains only finitely many terms that converge to zero as $T\to\infty$ almost surely, provided that $M(k,T)$ does not explode as $T\to\infty$. Notice that, according to  Lemma \ref{prop:ode_properties}-\ref{lem:distance-between-ode-solutions} and Lema \ref{rem:on-x-to-power-r}
\begin{align*}
M(k,T)  
 &\leq  \left\{ 
\begin{array}{cc}
\frac{\int_{0}^T \frac{1}{X_s}ds}{\int_0^T  \frac{1}{X_s^2}ds}\frac{1}{(u_0 \wedge \Eq^{(k-1)})}  & p \in (1,2)\\
\frac{ \int_{0}^T \frac{1}{X_s}ds}{\int_0^T  \frac{1}{X_s^2}ds}(p-1)\max\{\|{u}^{(k,T)}\|_\infty,\|{u}^{(k)}\|_\infty\}^{p-2} \frac{1}{(u_0 \wedge \Eq^{(k-1)})^{p-1}}   & p > 2.
\end{array}\right.
\end{align*}
Therefore, for $p\in(0,1)$, $M(k,T)$ depends on $T$ only through the quotient ${\int_{0}^T \frac{1}{X_s}ds}/{\int_0^T  \frac{1}{X_s^2}ds}$, which is almost surely convergent and hence bounded as $T\to\infty$. In the case of $p>1$, $M(k,T)$ also depends on $T$ thorough $\|{u}^{(k,T)}\|_\infty$, but thanks to Lemma \ref{prop:ode_properties}
$$
\|{u}^{(k,T)}\|_\infty \leq u_0 \lor \Eq^{(k,T)} = u_0 \lor \left(\frac{\hat{\theta}^{(k-1,T)}}{c}\right)^{1/p},
$$
and from Lemma \ref{lemma:a-priori-control-estimator-surrogate} we conclude that also in this case  $M(k,T)$ is bounded as $T\to \infty$, and hence
$$\lim_{T\to\infty}\left| \hat{\theta}^{(k,T)} -     \hat{\theta}^{(k)}\right|=0, \;a.s.$$
Finally, since
\begin{align*}
\left| \hat{\theta}^{(k,T)} -    \theta\right|  &= \left| \hat{\theta}^{(k,T)} -     \hat{\theta}^{(k)}\right|+\left|    \hat{\theta}^{(k)}-\theta\right|,
\end{align*}
is clear that
$$
\lim_{k\infty}\lim_{T\to\infty}\left| \hat{\theta}^{(k,T)} -    \theta\right|=0.
$$
To obtain the central limit theorem for $\hat{\theta}^{(k,T)} $, just notice that
$$
\hat{\theta}^{(k,T)} -\theta+\frac{c\int_{0}^T \frac{\uu_s^{p-1}-({u}_{s}^{(k,T)})^{p-1}}{X_s}ds}{\int_0^T  \frac{1}{X_s^2}ds}=\quo(T)-\theta+\frac{c\int_{0}^T \frac{\uu_s^{p-1}}{X_s}ds}{\int_0^T  \frac{1}{X_s^2}ds}
$$
from Lemma \ref{lem:convergence-of-QuoT}-\ref{lem:tclQ}, the results follows.

\end{proof}

\subsection{On the discretely sampled estimator }\label{app:on-discrete-estimator}

\begin{lemma}[A priori control of the estimator]\label{lemma:a-priori-control-estimator}
The estimator $\hat{\theta}^{(k,T,\dt)}$ is stochastically bounded in $\dt\to0$ in the sense that:
$$
\lim_{M\to\infty}\sup_{\dt>0} \P\left(\left|\hat{\theta}^{(k,T,\dt)}\right|\geq M\right)=0.
$$
\end{lemma}

\begin{proof}
Let us recall that 
\begin{equation}
    \hat{\theta}^{(k,T,\dt)} =\max\left\{  \frac{\sum_{i=0}^{n-1}\frac{\Delta X_{t_i}}{X_{t_i}^2} + c\sum_{i=0}^{n-1} \frac{({u}_{t_i}^{(k,T,\dt)})^{p-1}}{X_{t_i}}\Delta t_i}{  \sum_{i=0}^{n-1} X_{t_i}^{-2} \Delta t_i},\quad 0 \right\}.
\end{equation}
Then by triangular inequality and Lemma \ref{prop:ode_properties}:
\begin{equation*}
\begin{aligned}
  |\hat{\theta}^{(k,T,\dt)}|&\leq \left| \frac{\sum_{i=0}^{n-1}\frac{\Delta X_{t_i}}{X_{t_i}^2}}{  \sum_{i=0}^{n-1} X_{t_i}^{-2} \Delta t_i}\right|  +\left| \frac{ c\sum_{i=0}^{n-1} \frac{({u}_{t_i}^{(k,T,\dt)})^{p-1}}{X_{t_i}}\Delta t_i}{  \sum_{i=0}^{n-1} X_{t_i}^{-2} \Delta t_i}\right|     \\
  &\leq R_1(\dt)+R_2(\dt)\max(u_0,(\hat{\theta}^{(k-1,T,\dt)}/c)^{1/p})^{p-1},
\end{aligned}
\end{equation*}
where 
$$
R_1(\dt):= \frac{\left|\sum_{i=0}^{n-1}\frac{\Delta X_{t_i}}{X_{t_i}^2}\right|}{  \sum_{i=0}^{n-1} X_{t_i}^{-2} \Delta t_i}, \;\text{and}\;
R_2(\dt):= \frac{ c\sum_{i=0}^{n-1}X_{t_i}^{-1}\Delta t_i}{  \sum_{i=0}^{n-1} X_{t_i}^{-2} \Delta t_i}.
$$

Then, there exists a constant $C_p$ depending only on $p$ such that
\begin{equation*}
\begin{aligned}
  |\hat{\theta}^{(k,T,\dt)}|&\leq R_1(\dt)+C_p R_2(\dt)\left(\uu_0^{p-1}+(\hat{\theta}^{(k-1,T,\dt)}/c)^{(p-1)/p}\right)
\end{aligned}
\end{equation*}

Thus, from H\"older inequality we get for any $p\geq1$
\begin{equation}\label{eq:inductive-step-a-priori-bound-dt}
\begin{aligned}
  |\hat{\theta}^{(k,T,\dt)}|&\leq R_1(\dt)+C_p R_2(\dt)\left(1+\uu_0^{p-1}\right)+C_p R_2(\dt)\,  |\hat{\theta}^{(k-1,T,\dt)}|,
\end{aligned}
\end{equation}
where now $C_p$ depends on $p$ and $c$.

Iterating \eqref{eq:inductive-step-a-priori-bound-dt} $k$ times, we obtain 
$$
 |\hat{\theta}^{(k,T,\dt)}| \leq \left(R_1(\dt)+C_p R_2(\dt)\left(1+\uu_0^{p-1}\right)\right)\sum_{m=0}^{k-1}C_p^m\, R_2^m(\dt) + C_p^k\, R_2^k(\dt)\hat{\theta}^{(0,T,\dt)}.
$$
Notice that $R_1(\dt), R_2(\dt) $ are stochastically bounded, since both converge in probability to a finite random variable when $\dt\to0$. Then we conclude from the last bound that $ \hat{\theta}^{(k,T,\dt)} $ is stochastically bounded as it is dominated by the finite sum and product of stochastically bounded random variables.

\end{proof}

\begin{proof}[Proof of Proposition \ref{e1:convergence-in-dt} ] 
We want to study
$$\lim_{\dt\to0}\P\left( \left| \hat{\theta}^{(k,T,\dt)}- \hat{\theta}^{(k,T)}\right| > \epsilon , T\geq\TT  \right).$$
Notice that in the event $T\geq \TT$, $ \hat{\theta}^{(k,T)} > \cotaMin$ for all $k\geq1$. Notice also that $u^{(k,T,\dt)}$ depend implicit on $\dt$ and $T$ through $  \hat{\theta}^{(k-1,T,\dt)}  $. Additionally, since we are under Assumption \ref{hip:sign_theta},  $u^{(k,T,\dt)}$ is well defined for all $t\ge 0$. Then, 
\begin{equation}\label{eq:first-partition-convergence-1}
\begin{aligned}
 \left| \hat{\theta}^{(k,T,\dt)}- \hat{\theta}^{(k,T)}\right| &\leq \left|\frac{\sum_{i=0}^{n-1}\frac{\Delta X_{t_i}}{X_{t_i}^2}}{  \sum_{i=0}^{n-1} X_{t_i}^{-2} \Delta t_i} - c\frac{ \sum_{i=0}^{n-1} \frac{({u}_{t_i}^{(k,T,\dt)})^{p-1}}{X_{t_i}}\Delta t_i}{  \sum_{i=0}^{n-1} X_{t_i}^{-2} \Delta t_i} -\hat{\theta}^{(k,T)}\right|  \\
 &\leq  \left|\frac{\sum_{i=0}^{n-1}\frac{\Delta X_{t_i}}{X_{t_i}^2}}{  \sum_{i=0}^{n-1} X_{t_i}^{-2} \Delta t_i} - \frac{\int_0^T\frac{dX_t}{X_t^2}}{\int_0^T\frac{dt}{X_t^2}} \right|  
 + c\left|\frac{ \sum_{i=0}^{n-1} \frac{({u}_{t_i}^{(k,T,\dt)})^{p-1}}{X_{t_i}}\Delta t_i}{  \sum_{i=0}^{n-1} X_{t_i}^{-2} \Delta t_i} - \frac{\int_0^T\frac{(u_t^{(k,T)})^{p-1}}{X_t}dt}{\int_0^T\frac{dt}{X_t^2}} \right|\\
 &:= R_1(\dt)+ R_2(\dt,k).
\end{aligned}
\end{equation}

From the properties of stochastic integrals (see, e.g., \cite[Sec. I.7, II.4]{Protter1992}) and the continuous mapping theorem, $R_1(\dt)$ converges to zero in probability when $\dt$ converges to zero . Meanwhile, for $R_2(\dt,k)$ we have:
\begin{equation}\label{eq:second-deconposition-convergence-in-dt}
\begin{aligned}
 R_2(\dt,k)
  & \leq
  \left|\frac{1  }{  \sum_{i=0}^{n-1} X_{t_i}^{-2} \Delta t_i} \right|  \sum_{i=0}^{n-1} \frac{|({u}_{t_i}^{(k,T,\dt)})^{p-1}-({u}_{t_i}^{(k,T)})^{p-1}|}{X_{t_i}}\Delta t_i    + R_{2,1}(\dt,k),
\end{aligned}
\end{equation}
where 
\begin{align*}
R_{2,1}(\dt,k):=&\left|\frac{ \sum_{i=0}^{n-1} \frac{({u}_{t_i}^{(k,T)})^{p-1}}{X_{t_i}}\Delta t_i -\int_0^T\frac{(u_t^{(k,T)})^{p-1}}{X_t}dt }{  \sum_{i=0}^{n-1} X_{t_i}^{-2} \Delta t_i} \right| \\  
& \quad  +\left|\int_0^T\frac{(u_t^{(k,T)})^{p-1}}{X_t}dt\right|\left|\frac{ 1}{  \sum_{i=0}^{n-1} X_{t_i}^{-2} \Delta t_i} - \frac{1}{\int_0^T\frac{dt}{X_t^2}} \right|.
\end{align*}

Similar to previous arguments, for almost all $\omega\in\Omega$, $t\mapsto X_{t}(\omega)$ is a continuous and positive function, hence its Riemann sums converge to Riemann integrals when the size of the partition converges to zero (see, e.g., \cite[Sec. I.7]{Protter1992}). Therefore, for any $k\geq0$ and fixed $T>0$, $R_{2,1}(\dt,k)$ converges to zero almost surely.

We notice that, see Lemma \ref{rem:on-x-to-power-r} above for details, there exists  $L_p(t_i,k,T,\dt)$ such that
\[
|({u}_{t_i}^{(k,T,\dt)})^{p-1}-({u}_{t_i}^{(k,T)})^{p-1}| 
\leq L_p(t_i,k,T,\dt) |{u}_{t_i}^{(k,T,\dt)}-{u}_{t_i}^{(k,T)}|.
\]
Moreover,  thanks to Lemma \ref{prop:ode_properties}, for any $p\ge1$, $L_p(\cdot)$ can be bounded independent of the time $t_i$ (denoted by $L_p(k,T,\dt)$ in that case) given by
$$
L_p(k,T,\dt)  = \left\{
\begin{array}{cc}
1/\min\left\{u_0,(\hat{\theta}^{(k-1,T)}/c)^{1/p}\right\}^{2-p} ,&\text{ if }p\in(1,2)\\
1,&\text{ if }p=2\\
(p-1)\max\left\{u_0,(\hat{\theta}^{(k-1,T,\dt)}/c,(\hat{\theta}^{(k-1,T)}/c)^{1/p} \right\}^{p-2},&\text{ if }p>2.
\end{array}
\right.
$$
Notice that $ L_p(k,T,\dt)$ does not depend on $\dt$ if $p\leq 2$, and in the event $\{T\geq\TT\}$, $\hat{\theta}^{(k-1,T)}$ is bounded away from zero, then for $p\leq 2$, $ L_p(k,T,\dt)\ind{T\geq\TT}$ is stochastically bounded. In the case $p>2$, in Lemma \ref{lemma:a-priori-control-estimator} we prove that $\hat{\theta}^{(k-1,T,\dt)}$ is stochastically bounded in $\dt$, whereas  Lemma \ref{lemma:a-priori-control-estimator-surrogate} states that $\hat{\theta}^{(k-1,T)}$ is almost surely bounded, hence $L_p(k,T,\dt)$ is stochastically bounded too, when $p>2$.

Then, for the first term in the right-hand side of \eqref{eq:second-deconposition-convergence-in-dt} we have
\begin{align*}
 \left|\frac{1  }{  \sum_{i=0}^{n-1} X_{t_i}^{-2} \Delta t_i} \right|&  \sum_{i=0}^{n-1} \frac{|({u}_{t_i}^{(k,T,\dt)})^{p-1}-({u}_{t_i}^{(k,T)})^{p-1}|}{X_{t_i}}\Delta t_i\\ 
 & \leq     \left|\frac{1  }{  \sum_{i=0}^{n-1} X_{t_i}^{-2} \Delta t_i} \right|  \sum_{i=0}^{n-1} \frac{ L_p(k,T,\dt) |{u}_{t_i}^{(k,T,\dt)}-{u}_{t_i}^{(k,T)}|}{X_{t_i}}\Delta t_i.
\end{align*}
Using the bound \eqref{eq:distance-sols-ode} for the distance between solutions of the ODE, there exists\\
$C(\hat\theta^{(k-1,T)},\hat\theta^{(k-1,T,\dt)} ,u_0)>0$, such that we obtain
\begin{align*}
 &\left|\frac{1  }{  \sum_{i=0}^{n-1} X_{t_i}^{-2} \Delta t_i} \right|  \sum_{i=0}^{n-1} \frac{|({u}_{t_i}^{(k,T,\dt)})^{p-1}-({u}_{t_i}^{(k,T)})^{p-1}|}{X_{t_i}}\Delta t_i   \\
 & \qquad \leq \left|\frac{1  }{  \sum_{i=0}^{n-1} X_{t_i}^{-2} \Delta t_i} \right|  \sum_{i=0}^{n-1} \frac{\Delta t_i}{X_{t_i}}  C(\hat\theta^{(k-1,T)},\hat\theta^{(k-1,T,\dt)} ,u_0, \omega)| \hat{\theta}^{(k-1,T,\dt)}- \hat{\theta}^{(k-1,T)}|\\
 & \qquad =: M(k,\dt)| \hat{\theta}^{(k-1,T,\dt)}- \hat{\theta}^{(k-1,T)}|,
\end{align*}
where $M(k,\dt)$ is is stochastically bounded.  

Putting everything together in \eqref{eq:first-partition-convergence-1}, we will obtain:

$$
 \left| \hat{\theta}^{(k,T,\dt)}- \hat{\theta}^{(k,T)}\right|  \leq R_1(\dt) + R_{2,1}(k,\dt) + M(k,\dt)| \hat{\theta}^{(k-1,T,\dt)}- \hat{\theta}^{(k-1,T)}|,
$$
where for any $k$, $R_{2,1}(k,\dt)$ converges to zero almost surely, and $ R_1(\dt)$ converges to zero in probability. Iterating, and recalling $\hat{\theta}^{(0,T,\dt)} = \hat{\theta}^{(0,T)}=0$, we obtain
$$
 \left| \hat{\theta}^{(k,T,\dt)}- \hat{\theta}^{(k,T)}\right|  \leq R_1(\dt)\sum_{j=0}^{k-1}\prod_{l=0}^{j-1}M(k-l,\dt) + \sum_{j=0}^{k-1}R_{2,1}(k-j,\dt)\prod_{l=0}^{j-1}M(k-l,\dt),
$$
where an empty product is interpreted as $1$. Recall that $k$ is fixed, so in the right-hand side of the last inequality,  there are a finite number of products between random variables that are stochastically bounded and random variables that converge to zero almost surely or in probability. Therefore, the whole right-hand side converges to zero in probability, as desired.\\
\end{proof}

%%%%%%%%%%%%%%%%%%%%%%%%%%%%%%%%%%%%%
%%%%%%%%%%% BACK MATTER %%%%%%%%%%%%%%%%%%%

\paragraph{Acknowledgments:}
The authors gratefully acknowledge the hospitality of INRIA Université Côte d'Azur during a research visit, in the course of which part of this work was developed. We also thank Mireille Bossy for the constructive discussions during the preparation of this article.
\paragraph{Fundings:} 
The authors were partially supported by INRIA Associated Team SWAM and Programa Regional MathAmdSud AMSUD 240054, \emph{Stochasticity \& Chaos in Multiscale Phenomena}. H.O. has been also partially supported by FONDECYT Regular Nº1242001 and  ANID-Exploration 13220168. E.G. has been also partially supported by ANID-Doctoral Scholarship Nº21221664 and Adventist University of Chile PIR Nº274.  

\bibliographystyle{plain}
\bibliography{Bibliography}

\begin{thebibliography}{10}

\bibitem{Amorino2023}
C.~Amorino, A.~Heidari, V.~Pilipauskait{\.e}, and M.~Podolskij.
\newblock Parameter estimation of discretely observed interacting particle
  systems.
\newblock {\em Stochastic Processes and their Applications}, 158:1--40, 2023.

\bibitem{Baltazar2010}
F.~Baltazar-Larios and M.~S{\o}rensen.
\newblock Maximum likelihood estimation for integrated diffusion processes.
\newblock In {\em Contemporary Quantitative Finance}, pages 407--423. Springer,
  2010.

\bibitem{10.3150/bj/1116340291}
B.~Bibby, I.~Skovgaard, and M.~S{\o}rensen.
\newblock {Diffusion-type models with given marginal distribution and
  autocorrelation function}.
\newblock {\em Bernoulli}, 11(2):191 -- 220, 2005.

\bibitem{Bossy:2019aa}
M.~Bossy, J.~Fontbona, and H.~Olivero.
\newblock Synchronization of stochastic mean field networks of {Hodgkin-Huxley}
  neurons with noisy channels.
\newblock {\em Journal of Mathematical Biology}, 78(6):1771--1820, 2019.

\bibitem{Bossy2022}
M.~Bossy, J.~Jabir, and K.~Mart{\'\i}nez.
\newblock Instantaneous turbulent kinetic energy modelling based on lagrangian
  stochastic approach in cfd and application to wind energy.
\newblock {\em Journal of Computational Physics}, 464:110929, 2022.

\bibitem{CarmonaDelarue2018}
R.~Carmona and F.~Delarue.
\newblock {\em Probabilistic Theory of Mean Field Games with Applications I}.
\newblock Springer, 2018.

\bibitem{CarmonaDelarue2018b}
R.~Carmona and F.~Delarue.
\newblock {\em Probabilistic Theory of Mean Field Games with Applications {II}:
  {M}ean Field Games with Common Noise and Master Equations}.
\newblock Springer, 2018.

\bibitem{ChaintronDiez2022a}
L.~Chaintron and A.~Diez.
\newblock Propagation of chaos: {A} review of models, methods and applications.
  {I}. {M}odels and methods.
\newblock {\em Kinetic and Related Models}, 15(6):895--1015, 2022.

\bibitem{ChaintronDiez2022b}
L.~Chaintron and A.~Diez.
\newblock Propagation of chaos: {A} review of models, methods and applications.
  {II}. {A}pplications.
\newblock {\em Kinetic and Related Models}, 15(6):1017--1173, 2022.

\bibitem{chen2025}
X.~Chen, G.~dos Reis, and W.~Stockinger.
\newblock Wellposedness, exponential ergodicity and numerical approximation of
  fully super-linear mckean-vlasov sdes and associated particle systems.
\newblock {\em Electronic Journal of Probability}, 30(23):1--50, 2025.

\bibitem{colombani2023propagation}
L.~Colombani and P.~Le~Bris.
\newblock Propagation of chaos in mean field networks of fitzhugh-nagumo
  neurons.
\newblock {\em Mathematical Neuroscience and Applications}, 3, 2023.

\bibitem{Cordero:2026aa}
F.~Cordero, Ch. Jorquera, H.~Olivero, and L.~Videla.
\newblock Wright--fisher kernels: from linear to non-linear dynamics,
  ergodicity and mckean--vlasov scaling limits.
\newblock {\em Electronic Journal of Probability}, 31:1--49, 2026.

\bibitem{crevat2019mean}
J.~Crevat.
\newblock Mean-field limit of a spatially-extended fitzhugh-nagumo neural
  network.
\newblock {\em Kinetic and Related Models}, 12(6):1329--1358, 2019.

\bibitem{dohnal1987estimating}
G.~Dohnal.
\newblock On estimating the diffusion coefficient.
\newblock {\em Journal of Applied Probability}, 24(1):105--114, 1987.

\bibitem{FournieTalay1991}
{\'E}.~Fourni{\'e} and D.~Talay.
\newblock Application de la statistique des diffusions {\`a} un mod{\`e}le de
  taux d'int{\'e}r{\^e}t.
\newblock {\em Finance}, 12(2):79--111, 1991.

\bibitem{GenonCatalotLaredo2021}
V.~Genon-Catalot and C.~Laredo.
\newblock Parametric inference for small variance and long time horizon
  {M}c{K}ean--{V}lasov diffusion models.
\newblock {\em Electronic Journal of Statistics}, 15(2):5811--5854, 2021.

\bibitem{hausler2015stable}
E.~H{\"a}usler and H.~Luschgy.
\newblock {\em Stable convergence and stable limit theorems}, volume~74.
\newblock Springer, 2015.

\bibitem{HU20101030}
Y.~Hu and D~Nualart.
\newblock Parameter estimation for fractional ornstein--uhlenbeck processes.
\newblock {\em Statistics \& Probability Letters}, 80(11):1030--1038, 2010.

\bibitem{karatzas1991}
I.~Karatzas and S.~Shreve.
\newblock {\em Brownian Motion and Stochastic Calculus}.
\newblock Springer, New York, NY, 2nd edition, 1991.

\bibitem{kumar2021explicit}
C.~Kumar and Neelima.
\newblock On explicit milstein-type scheme for mckean--vlasov stochastic
  differential equations with super-linear drift coefficient.
\newblock {\em Electronic Journal of Probability}, 26:1--32, 2021.

\bibitem{kumar2022well}
C.~Kumar, Neelima, C.~Reisinger, and W.~Stockinger.
\newblock Well-posedness and tamed schemes for mckean--vlasov equations with
  common noise.
\newblock {\em The Annals of Applied Probability}, 32(5):3283--3330, 2022.

\bibitem{lambert2026evolution}
A.~Lambert, H.~Leman, H.~Morlon, and J.~Tchouanti.
\newblock Evolution of a trait distributed over a large fragmented population:
  Propagation of chaos meets adaptive dynamics.
\newblock {\em Journal of Mathematical Biology}, 92(5):70, 2026.

\bibitem{LiuQiao2022}
M.~Liu and H.~Qiao.
\newblock Parameter estimation of path-dependent {M}c{K}ean--{V}lasov
  stochastic differential equations.
\newblock {\em Acta Mathematica Scientia}, 42(3):876--886, 2022.

\bibitem{EvaLocherbach2015}
E.~L{\"o}cherbach.
\newblock Ergodicity and speed of convergence to equilibrium for diffusion
  processes.
\newblock Technical report, CNRS UMR 8088, D{\'e}partement de
  Math{\'e}matiques, Universit{\'e} de Cergy-Pontoise, November 2015.
\newblock Lecture Notes.

\bibitem{mao2007stochastic}
X.~Mao.
\newblock {\em Stochastic differential equations and applications}.
\newblock Elsevier, 2007.

\bibitem{McKean1966}
H.~McKean.
\newblock A class of markov processes associated with nonlinear parabolic
  equations.
\newblock {\em Proceedings of the National Academy of Sciences of the United
  States of America}, 56(6):1907--1911, 1966.

\bibitem{Meleard:1996aa}
S.~M{\'e}l{\'e}ard.
\newblock Asymptotic behaviour of some interactive particle systems;
  {M}c{K}ean-{V}lasov and {B}oltzmann models.
\newblock {\em Probabilistic Models for Nonlinear Partial Differential
  Equations}, 1627:42--95, 1996.

\bibitem{Miao2004}
W.~Miao.
\newblock {\em Quadratic variation estimators for diffusion models in finance}.
\newblock PhD thesis, University Park Los Angeles, USA, 2004.
\newblock AAI3140520.

\bibitem{Pavliotis2022}
G.~Pavliotis and A.~Zanoni.
\newblock Eigenfunction martingale estimators for interacting particle systems
  and their mean field limit.
\newblock {\em SIAM Journal on Applied Dynamical Systems}, 21(4):2338--2370,
  2022.

\bibitem{Pavliotis2024}
G.~Pavliotis and A.~Zanoni.
\newblock A method of moments estimator for interacting particle systems and
  their mean field limit.
\newblock {\em SIAM/ASA Journal on Uncertainty Quantification}, 12(2):262--288,
  2024.

\bibitem{Pavliotis2025}
G.~Pavliotis and A.~Zanoni.
\newblock Linearization of ergodic {McKean} {SDEs} and applications.
\newblock {\em Nonlinearity}, 38(8):085008, 2025.

\bibitem{perko2013differential}
L.~Perko.
\newblock {\em Differential equations and dynamical systems}, volume~7.
\newblock Springer Science \& Business Media, 2013.

\bibitem{pope1994lagrangi}
S.B. Pope.
\newblock Lagrangi-an pdf methods for turbulent flows.
\newblock {\em Annu. Rev. Fluid Mech}, 23:63, 1994.

\bibitem{Protter1992}
P.~Protter.
\newblock {\em Stochastic Integration and Differential Equation}.
\newblock Springer-Verlag, Berlin, Heidelberg, second edition, 1992.

\bibitem{quininao2020clamping}
C.~Quininao and J.~Touboul.
\newblock Clamping and synchronization in the strongly coupled fitzhugh--nagumo
  model.
\newblock {\em SIAM journal on applied dynamical systems}, 19(2):788--827,
  2020.

\bibitem{RenWu2021}
P.~Ren and J.~Wu.
\newblock Least squares estimation for path-distribution dependent stochastic
  differential equations.
\newblock {\em Applied Mathematics and Computation}, 410:126457, 2021.

\bibitem{Sharrock2023}
L.~Sharrock, N.~Kantas, P.~Parpas, and G.~Pavliotis.
\newblock Online parameter estimation for the {M}c{K}ean--{V}lasov stochastic
  differential equation.
\newblock {\em Stochastic Processes and their Applications}, 162:481--546,
  2023.

\bibitem{zbMATH00195091}
A.~Skorokhod.
\newblock {\em Asymptotic methods in the theory of stochastic differential
  equations. {Transl}. from the {Russian} by {H}. {H}. {McFaden}, transl. ed.
  by {Ben} {Silver}}, volume~78 of {\em Transl. Math. Monogr.}
\newblock Providence, RI: American Mathematical Society, 1989.

\bibitem{Sznitman1991}
A.~Sznitman.
\newblock Topics in propagation of chaos.
\newblock In {\em \'Ecole d'\'Et\'e de Probabilit\'es de Saint-Flour
  XIX---1989}, volume 1464 of {\em Lecture Notes in Mathematics}, pages
  165--251. Springer, 1991.

\bibitem{videla2025persistence}
L.~Videla, M.~Tejo, C.~Qui{\~n}inao, P.~Marquet, and R.~Rebolledo.
\newblock Persistence and neutrality in interacting replicator dynamics.
\newblock {\em Journal of Mathematical Biology}, 90(2):15, 2025.

\bibitem{Wen2016}
J.~Wen, X.~Wang, S.~Mao, and X.~Xiao.
\newblock Maximum likelihood estimation of mckean--vlasov stochastic
  differential equation and its application.
\newblock {\em Applied Mathematics and Computation}, 274:237--246, 2016.

\end{thebibliography}
\end{document}